\documentclass[reqno,a4paper]{amsart}
\usepackage{hyperref}
\usepackage{amsfonts}
\usepackage{amsthm,amsmath}
\usepackage[integrals]{wasysym}
\usepackage{enumerate}
\usepackage{fancyhdr}
\usepackage{amssymb}
\usepackage{chemarrow} 
\usepackage{tikz}
\usepackage{mathtools}
\usepackage{mathrsfs}

\usepackage{comment}

\usepackage{enumitem}

\newtheorem{theorema}{Theorem}

\newtheorem{theoremabis}{Theorem}

\numberwithin{equation}{section}
\newtheorem{theorem}{Theorem}[section]
\newtheorem{lemma}[theorem]{Lemma}

\newtheorem{proposition}[theorem]{Proposition}

\theoremstyle{remark}
\newtheorem{remark}{Remark}
\allowdisplaybreaks

\definecolor{darkgreen}{rgb}{0,0.5,0}
\definecolor{darkblue}{rgb}{0,0,0.7}
\definecolor{darkred}{rgb}{0.9,0.1,0.1}
\definecolor{lightblue}{rgb}{0,0.51,1}

\def\cB{{\mathcal B}}

\def\div{ \hbox{\rm div}  }

\def\N{{\mathbb N}}
\def\R{{\mathbb R}}
\def\T{{\mathbb T}}
\def\Z{{\mathbb Z}}

\begin{document}
\title[\hfilneg \hfil ]
{On symmetry breaking for the Navier-Stokes equations}

\author[T. Barker]{ Tobias Barker}
 \author[C. Prange]{Christophe Prange}
 \author[J. Tan]{Jin Tan}

\address[T. Barker]{\newline Department of Mathematical Sciences, University of Bath, Bath BA2 7AY. UK}
\email{tobiasbarker5@gmail.com}

\address[C. Prange]{\newline  Laboratoire de Math\'{e}matiques AGM, UMR CNRS 8088, Cergy Paris Universit\'{e}, 2 Avenue Adolphe Chauvin, 95302     Cergy-Pontoise Cedex, France}
\email{christophe.prange@cyu.fr}

\address[J. Tan]{\newline  Laboratoire de Math\'{e}matiques AGM, UMR CNRS 8088, Cergy Paris Universit\'{e}, 2 Avenue Adolphe Chauvin, 95302     Cergy-Pontoise Cedex, France}
\email{jin.tan@cyu.fr}

\date{\today}

\begin{abstract}
Inspired by an open question by Chemin and Zhang 
about the re\-gularity of the 3D Navier-Stokes equations with one initially small component, we investigate symmetry breaking and symmetry preservation. 
Our results fall in three classes. First we prove strong symmetry breaking. Speci\-fically, we demonstrate third component norm inflation (3rdNI) and Isotropic Norm Inflation (INI) starting from zero third component. Second we prove symmetry breaking for initially zero third component, even in the presence of a favorable initial pressure gradient. Third we study certain symmetry preserving solutions with a shear flow structure. Specifically, we give applications to the inviscid limit and exhibit explicit solutions that inviscidly damp to the Kolmogorov flow.
\end{abstract}

\maketitle

\section{Introduction}
Symmetries preserved by evolution play an important role in the mathematical theory of the Navier-Stokes equations and Euler equations:
\begin{equation}\label{NS} 
\begin{aligned} 
 \partial_t{  u^{\nu}}+u^{\nu}\cdot\nabla u^{\nu}+\nabla   P&=\nu\Delta   u^{\nu}\quad&\hbox{in }   \R_+ \times\T^3,\quad
 \div\, u^{\nu}&=0,\quad\nu\geq 0.
\end{aligned}
\end{equation}
On the one hand, certain preserved symmetries lead to the preservation of certain structures that grant smoothness of solutions \cite{Lady68}, \cite{UI68}. On the other hand, preserved symmetries reduce the number of degrees of freedom of the Navier-Stokes and Euler equations, which can make it possible to prove or numerically investigate the existence of singularities \cite{Elg21}, \cite{CH22}, \cite{hou2022potential,hou2022potentially}.

In this vein, in recent years there has been a substantial amount of activity aimed at showing that additional assumptions of one component of the velocity field (solving the Navier-Stokes equations) imply that the solution is regular. On the other side of coin, this corresponds to showing that solutions of the Navier-Stokes equations that become singular must do so in an isotropic manner. Research in this direction was initiated in the seminal paper of Neustupa and Penel \cite{NP99}. Since then there have been many contributions to one-component regularity for the Navier-Stokes equations, with recent contributions showing regularity provided that one-component of the velocity field has a finite norm either almost preserved \cite{CW21,CGZ19} or preserved\footnote{Currently one component regularity criteria in terms of norms preserved with respect to the Navier-Stokes rescaling, involve spatial norms with some differentiability or Lorentz time norms. It remains a long standing open problem if a solution to the Navier-Stokes equations $v$, with third component $v_{3}\in L^{q}(0,T; L^{p}(\mathbb{R}^3))$ $(\frac{3}{p}+\frac{2}{q}=1,\,p\in [3,\infty])$ is smooth on $\mathbb{R}^3\times (0,T]$.} with respect to the Navier-Stokes scaling symmetry \cite{CZ16,CZZ,WWZ20}.

The purpose of this paper is to understand the dynamics of the Navier-Stokes equations when one-component of the initial data is zero. Throughout we will set the third component of the initial data to be zero, without loss of generality. Our main motivation is an open question raised by Chemin, Zhang and Zhang \cite{CZZ} when discussing endpoint one-component regularity criteria. Specially, in \cite[page 873]{CZZ}, Chemin, Zhang and Zhang formulate the following open question:
\begin{quote}
\textbf{(Q)}\ [If] for some unit vector $e$ of $\R^3$, [the component of the initial data]  $\|u_{\rm in}\cdot e\|_{H^\frac12}$ is small with respect to some universal constant, is it implied that there is no blow 
up for the Fujita-Kato solution of (NS)?
\end{quote}

\subsection{Main results of the paper}
In relation to the aforementioned open problem \textbf{(Q)}, our first two results show that initial data with zero third component can exhibit third component norm inflation (3rdNI, Theorem \ref{Th-3}) and Isotropic Norm Inflation (INI, Theorem \ref{Th-3bis}) with respect to critical norms specified in \cite{BP08}. 
\begin{theorema}[strong symmetry breaking]\label{Th-3}
  For any $ 0<\delta<1$, there exists mean-free $C^\infty(\T^3)$ solenoidal initial data $u_{\rm in}$ and $\bar u_{\rm in}$\footnote{This data has the structure given \eqref{3D-nu-data}. Our result shows that the solution map is not continuous at $\bar u_{\rm in}$ in the critical space $\dot{B}^{-1}_{\infty, \infty}$.} with vanishing third component, 
\begin{align*}
     \|u_{\rm in}-\bar u_{\rm in}\|_{\dot{B}^{-1}_{\infty, \infty}} =\|u_{\rm in}^{\rm h}-\bar u_{\rm in}^{\rm h}\|_{\dot{B}^{-1}_{\infty, \infty}}<\delta,
\end{align*}
and such that the following holds true.\\
There exists a unique solution $u$ $({\rm resp. } ~\bar u)$ of the Cauchy problem \eqref{NS} subject to initial data $u_{\rm in}$ (resp. $\bar u_{\rm in}$) belonging to $C^\infty((0, T]\times \T^3)$ for some time $0<T<\delta$ with $\bar u^3\equiv0$ on $[0,T]\times\R^3$ and 
\begin{align*}
    \|u^3{(T,\cdot)}-\bar u^3{(T,\cdot)}\|_{\dot{B}^{-1}_{\infty, \infty}}=  \|u^3{(T,\cdot)}\|_{\dot{B}^{-1}_{\infty, \infty}}>\frac{1}{\delta}.
\end{align*}
Moreover,
\begin{align}\label{S4-eq}
     \min\big(\|u^1{(T,\cdot)}\|_{L^3},\|u^2{(T,\cdot)}\|_{L^3},\|u^3{(T,\cdot)}\|_{L^3}\big)>\frac{1}{\delta}.
\end{align}
\end{theorema}
{Theorem \ref{Th-3} does not provide negative evidence towards \textbf{(Q)}, but demonstrates that regularity in that case can only be granted by a yet to be discovered mechanism unrelated to the preservation of smallness of the third component of the corresponding solution. However, the solution in Theorem \ref{Th-3} remains small in $\dot{B}^{-1}_{\infty,\infty}$ at $T$ in certain directions (see the discussion in subsection \ref{sec.heuriso}). Thus, the construction in Theorem \ref{Th-3} does not rule out the possibility that solutions, with initial third component equal to zero, remain small along some time-varying direction. Such a possibility is in fact ruled out by our second result below.}
{\begin{theoremabis}[strong isotropic symmetry breaking]\label{Th-3bis}
 For any $ 0<\delta<1$, there exists mean-free $C^\infty(\T^3)$ solenoidal initial data $u_{\rm in}$ and $\bar u_{\rm in}$\footnote{This data has the structure given \eqref{3D-nu-data-bis}.} with vanishing third component, 
\begin{align*}
     \|u_{\rm in}-\bar u_{\rm in}\|_{\dot{B}^{-1}_{\infty, \infty}} =\|u_{\rm in}^{\rm h}-\bar u_{\rm in}^{\rm h}\|_{\dot{B}^{-1}_{\infty, \infty}}<\delta,
\end{align*}
and such that the following holds true.\\
There exists a unique solution $u$ of the Cauchy problem \eqref{NS} subject to initial data $u_{\rm in}$ belonging to $C^\infty((0, T]\times \T^3)$ for some time $0<T<\delta$ and such that 
\begin{align}\label{e.sini}
    {\inf_{\mathbf{e}\in\mathbb{R}^3:|\mathbf{e}|=1}\|u(T,\cdot)\cdot\mathbf{e}\|_{\dot{B}^{-1}_{\infty, \infty}}>\frac{1}{\delta}.}
\end{align}
\end{theoremabis}}
{We dub the norm inflation in all directions in \eqref{e.sini} `Isotropic Norm Inflation' (INI).}

Now define the initial pressure $P_{\rm in}$ associated to the initial data $u_{\rm in}$, which satisfies
\begin{equation}\label{presideqn}
-\Delta P_{\rm in}:=   \nabla_{} u_{\rm in} :(\nabla u_{\rm in})^{\rm T}. 
\end{equation}
Note that the initial data {used to prove Theorems \ref{Th-3} and \ref{Th-3bis}, which will heuristically be described in subsections \ref{sec.heur}-\ref{sec.heuriso}, both necessarily generate} an initial pressure $P_{\rm in}$ that satisfies $\partial_{3}P_{\rm in}\neq 0.$ From the equation for the third component of the associated solution \eqref{NS}, it is qualitatively clear that such an initial pressure will always produce a solution that breaks the symmetry of the third component zero. In this regard, we call pressure of this type \textit{unfavorable}.\footnote{\label{foot.unfavorable}The terminology \textit{unfavorable} refers here to the fact that the pressure is unfavorable to symmetry preservation.} Notice that there are other examples of plane-wave initial data that demonstrate symmetry breaking. We refer for instance to Figure \ref{fig.TG} that shows breaking for the Taylor-Green vortex 
$$
u_{in}(x_1,x_2,x_3)=(\sin x_1 \cos x_2 \cos x_3,-\cos x_1 \sin x_2 \cos x_3,0).
$$
Notice that 
$$
\Delta P_{in}=2(\cos x_3)^2\big((\cos x_1 \cos x_2)^2-(\sin x_1 \sin x_2)^2\big)
$$
so that the pressure for the Taylor-Green vortex is also unfavorable.

\begin{figure}[h]
\begin{center}
\includegraphics[scale=.4]{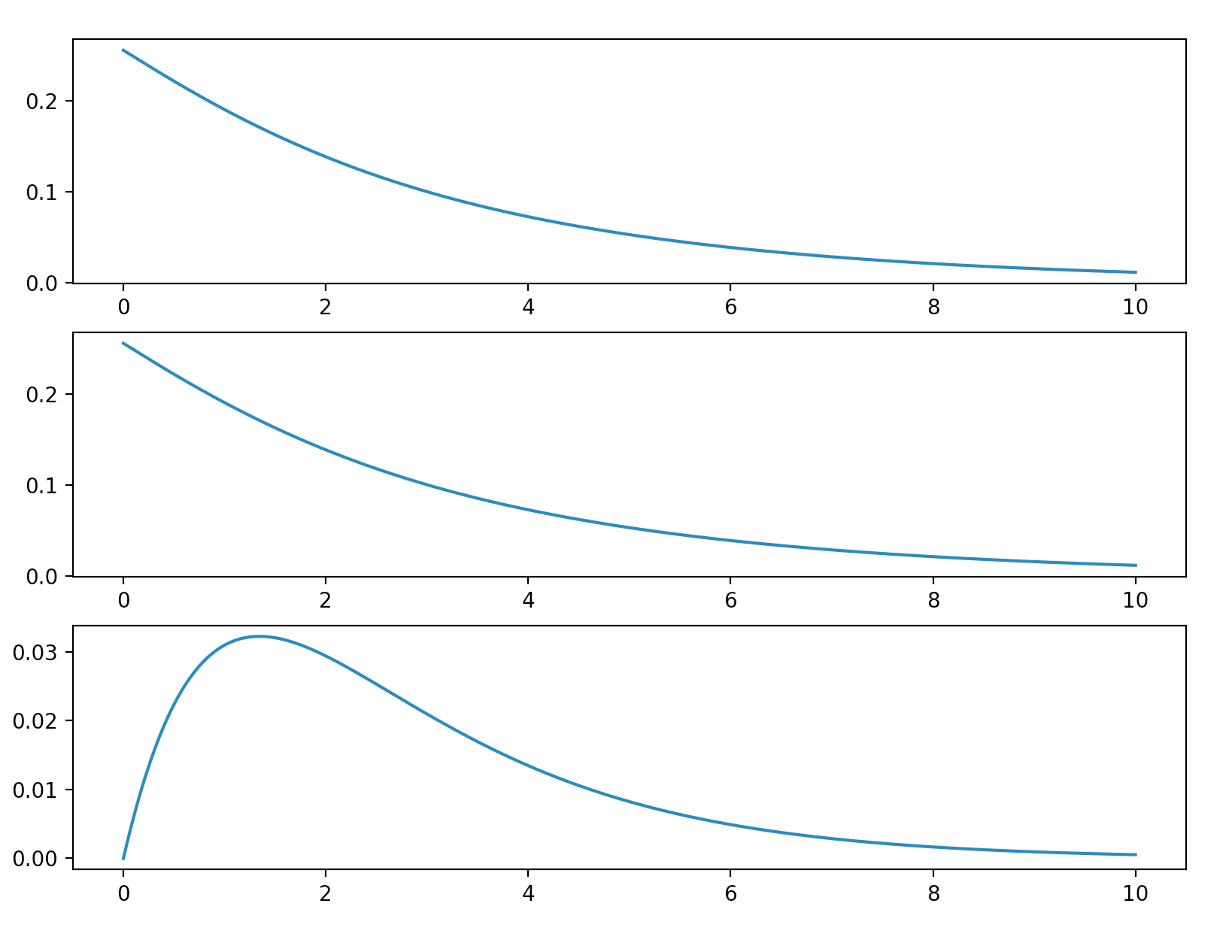}
\caption{Taylor-Green vortex solution of Navier-Stokes with viscosity $\nu=10^{-1}$. From top to bottom: $\|u^1(\cdot,t)\|_{L^1(\T^3)}$, $\|u^2(\cdot,t)\|_{L^1(\T^3)}$ and $\|u^3(\cdot,t)\|_{L^1(\T^3)}$. This simulation shows breaking with initial pressure unfavorable (see Footnote \ref{foot.unfavorable} for a definition)  
to symmetry preservation. Choice of parameters: total time $T=10$ and time step $dt=10^{-2}$; spectral code by Mikael Mortensen taken from \texttt{https://github.com/spectralDNS/spectralDNS} with $(2^5)^3$ mesh points.}
\label{fig.TG}
\end{center}
\end{figure}

As a consequence of the above, for initial data $u_{\rm in}$ with zero third component, we say that an initial pressure $P_{\rm in}$ is \emph{favorable}\footnote{\label{foot.favorable}The terminology \textit{favorable} refers here to the fact that the pressure is favorable to symmetry preservation. We demonstrate, see Theorem \ref{Th-2} that favorable pressure is still not enough to preserve the vanishing of the third component of the velocity, if it is zero initially.} if 
\begin{align}\label{Structure-initialdata}
\div_{\rm h} u^{\rm h}_{\rm in}=0\quad{\rm with}\quad \partial_3 P_{\rm in}=0, \quad{\rm where}~ -\Delta P_{\rm in}:=   \nabla_{\rm h} u^{\rm h}_{\rm in} :(\nabla_{\rm h} u^{\rm h}_{\rm in})^{\rm T}.
\end{align}
For a \textit{favorable} initial pressure, the equation for the third component of \eqref{NS} does not immediately imply that the third component of the solution breaks symmetry and becomes non-zero. In the Theorem below we are able to demonstrate an initial data below, which has zero third component and favorable initial pressure, yet the corresponding solution  breaks the symmetry and has non-zero third component on some time interval.

\begin{theorema}[symmetry breaking despite favorable pressure gradient]\label{Th-2}
We consider the initial data
\begin{equation}\label{initial-Growth-3D} 
 u_{\rm in}{=}   \left(  
            \cos x_2\, \frac{N}{N+\sin x_3},\,
   \cos x_1\,  \frac{N+\sin x_3}{N},\,
         0\right) 
\end{equation}
which has favorable initial pressure gradient in the sense that $\partial_3P_{in}=0$, see \eqref{Structure-initialdata}.\\
 Then there exists a positive constant    $N_0$ such that for any  $N>N_0,$ the initial data given by \eqref{initial-Growth-3D} generates a unique solution $u$ to the Navier-Stokes equations \eqref{NS} on $[0,1]\times\T^3$ that satisfies
\begin{equation}\label{e.loweru3}
  \|u^3(t, \cdot)\|_{L^\infty(\T^3)} \sim  \frac{t^2}{N^2},
\end{equation}
 for any   $0\leq t\leq  \frac{1}{N^2}\ll1$.
\end{theorema}

\begin{figure}[h]
\begin{center}
\includegraphics[scale=.4]{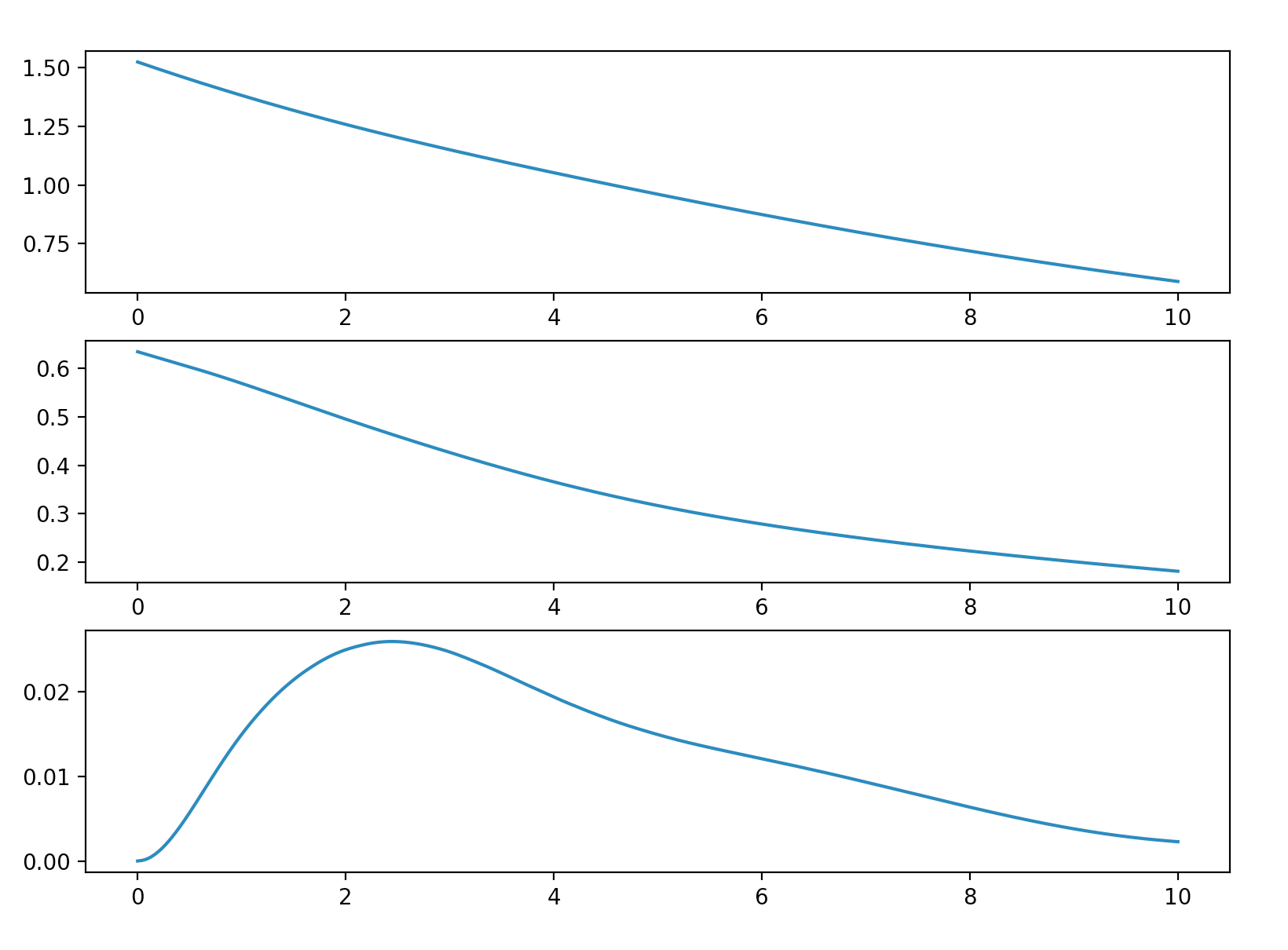}
\caption{ 
Simulation of the Navier-Stokes equations with viscosity $\nu=10^{-1}$ and initial data from Theorem \ref{Th-2} taking $N=1.1$. From top to bottom: $\|u^1(\cdot,t)\|_{L^1(\T^3)}$, $\|u^2(\cdot,t)\|_{L^1(\T^3)}$ and $\|u^3(\cdot,t)\|_{L^1(\T^3)}$. This simulation shows breaking despite initial pressure favorable (see Footnote \ref{foot.favorable} for a definition) 
to symmetry preservation. Choice of parameters: total time $T=10$ and time step $dt=10^{-2}$; spectral code by Mikael Mortensen taken from \texttt{https://github.com/spectralDNS/spectralDNS} with $(2^5)^3$ mesh points.}
\label{fig.ourdata}
\end{center}
\end{figure}

\begin{remark}[comparison to other symmetry breaking results]
The non-uniqueness numerical results of Guillod and {{\v{S}}ver{\'a}k} \cite{GS17} concern Leray-Hopf solutions of the Navier-Stokes equations that break a symmetry class. In the context of the Euler equations and convex integration, symmetry breaking and restoration mechanisms were explicitly investigated in \cite{bardos2013stability}. The non-uniqueness results for dissipative solutions of Euler by Scheffer \cite{Sch93}, Shnirelman \cite{Shni97}, De Lellis and Sz\'{e}kelyhidi \cite{DLS09}, Isett \cite{isett2018proof} and for weak solutions of Navier-Stokes by Buckmaster and Vicol \cite{BV19} can also be seen as symmetry breaking results. Our results are in a different vein though. We show breaking of symmetry on some time interval where the solution is unique and smooth.
\end{remark}
{\begin{remark}[isotropic motion with initial third component zero]
For the construction in Theorem \ref{Th-3bis}, one can also show that for all $\mathbf{e}\in\mathbb{R}^3$ with $|\mathbf{e}|=1$:
$$\|u(T,\cdot)\cdot \mathbf{e}\|_{\dot{B}^{-1}_{\infty,\infty}}\sim\frac{1}{\delta}. $$
Thus at time $T$, the velocity field is comparable in all unit directions with respect to the $\dot{B}^{-1}_{\infty,\infty}$ norm, despite evolving from initial data with zero third component.
\end{remark}}

\subsection*{Further results}

Our interest here is on flows that preserve the symmetry $u_3=0$ and on applications of such flows. 

There are indeed flows solving Navier-Stokes and Euler that have a shear-flow structure and keep the third component identically zero, as for instance the plane parallel channel flows $(u^1(x_3,t), u^{2}(x_1,x_3,t),0)$ introduced by Wang \cite{wang2001kato}. Rotating these flows gives rise to a whole family of symmetry preserving pressureless shear flows that we dub `2.75D shear flows', which are defined as follows. 
Let  $\lambda\in \mathbb{Z}$ be a constant. Consider the initial data
\begin{align}\label{3D-nu-data}
u_{\rm in}=\big(f(\lambda x_1+x_2, x_3), \,-\lambda f(\lambda x_1+x_2, x_3)-g(x_3), 0 \big),\quad x\in \T^3.
\end{align}
In fact, in the case that $u_{\rm in}$ merely belongs to $L^2(\T^3)$
(that corresponds to $f$ or $ g$ being rough),  one can obtain a Leray-Hopf weak solution to the problem \eqref{NS} with $\nu>0$, and initial data $u_{\rm in}$, given by
\begin{align}\label{3D-nu-solu}
u^\nu=\big(F_\nu(t, \lambda x_1+x_2, x_3), -\lambda F_\nu(t, \lambda x_1+x_2, x_3)-e^{\nu t\partial_{3}^2}g(x_3), 0\big)
\end{align}
where $F_{\nu}: \R_+\times \T^2 \to  \R$  is the unique global-in-time solution to
 \begin{equation}\label{DDheat}
\left\{\begin{aligned}
  &\partial_t F_{\nu}  -e^{\nu t\partial_{2}^2}g (y_2)\,\partial_1 F_{\nu}  = \nu\bigl((\lambda^2+1)\partial_{1}^2+ \partial_{2}^2\bigr) F_{\nu} \quad {\rm in} ~~\T^2\times\R_+ \\
 & F_{\nu} (0, y_1, y_2)=f(y_1, y_2).
\end{aligned}\right.
\end{equation}
For more insights about the derivation of these flows, we refer to Appendix \ref{S3a}.

\begin{remark}[2.75D shear flows for Euler]\label{rem.2.75Euler}
One also has `2.75D shear flows' that solve the Euler equations in a distributional sense:
\begin{equation}\label{3D-Euler-solu}
u^{E}(t, x)=\big(f(\lambda x_1+x_2+tg(x_3), x_3), -\lambda f(\lambda x_1+x_2+tg(x_3), x_3)-g(x_3), 0\big),
\end{equation}
where $f\in L^2(\T^3)$ and $g\in C^\infty(\T^2)$.
\end{remark}

\subsubsection*{Inviscid damping}

In Section \ref{S3-s} we show that the 2.75D shear flows for Euler, see Remark \ref{rem.2.75Euler}, inviscidly damp to the \textit{Kolmogorov flow} $u^{K}=(0,\sin x_3,0)$. The Kolmogorov flow is a stationary solution of the 3D Euler equations in $\mathbb{T}^3$. In \cite{CEW} Coti-Zelati, Elgindi and Widmayer exhibit 2D stationary solutions to the Euler equations near $u^{K}$, thus demonstrating a lack of inviscid damping near $u^{K}$. On the other hand, 2.75D shear Euler flows \eqref{3D-Euler-solu} can be used to produce explicit solutions that inviscidly damp\footnote{We thank Hao Jia for this observation.} to $u^{K}$ for large times. This and \cite{CEW} show that dynamics near the Kolmogorov flow in $\mathbb{T}^3$ are rich and no generic behavior can be expected. 

\subsubsection*{Vanishing viscosity}

In Section \ref{s-2.75limit} we investigate the vanishing viscosity limit for 2.75D shear flow solutions of Navier-Stokes that are Onsager supercritical. Turbulence theory  from \cite{K1,K2,K3} predicts that if $u^{\nu}$ is a weak Leray-Hopf solution in $\mathbb{T}^3\times (0,\infty)$, with viscosity $\nu$ and initial data $u_{\rm in}$, then generically one has \textit{anomalous dissipation}:
\begin{equation}\label{anondiss}
\liminf_{\nu\rightarrow 0} \nu\int\limits\limits_{0}^{T}\int\limits\limits_{\mathbb{T}^3} |\nabla u^{\nu}|^2 dxds>0.
\end{equation}
It is known that if \eqref{anondiss} and the vanishing viscosity limit holds in suitable topology, then the corresponding Euler flow $u^{E}$ must belong to \textit{Onsager supercritical spaces} such as $C^{\frac{1}{3}-}$ or $\dot{H}^{\frac{5}{6}-}$. See \cite{CET} and \cite{CCFS}, for example. 

Using 2.75D shear flows for the Navier-Stokes and Euler equations, we show in Propositions \ref{Th-4} and \ref{Th-5} that the vanishing viscosity limit and the corresponding Euler flow belonging to \textit{Onsager supercritical spaces} are not sufficient conditions for \textit{anomalous dissipation}. Moreover in Proposition \ref{Th-4}, we build upon the work of \cite{BT} to give an example of a rough solution to the 3D Euler equations that satisfies the local energy balance.

\subsection{Heuristics for strong symmetry breaking}
\label{sec.heur}

In this subsection, we give some heuristics for Theorem \ref{Th-3}.

The mechanism to get norm inflation in the critical $\dot B^{-1}_{\infty,\infty}$ space is well understood thanks to the work of Bourgain and Pavlovi\'{c} \cite{BP08}, and later Yoneda \cite{Y10} and Cheskidov and Dai \cite{CD14}. We mention here also the work of Wang \cite{Wa15}, which demonstrates norm inflation phenomena in the spaces  $\dot{B}^{-1}_{\infty, q}$ for $1\leq q\leq 2$, but the construction is different from the one considered here. 

Our point here is to explain how to get {norm inflation on the third component (3rdNI), starting from data with third component equal to zero as in the case of Theorem \ref{Th-3}}. Such {norm inflation on the third component} cannot be obtained from the previously known constructions.

As a starting point, let us consider the general plane-wave initial data 
\begin{equation}\label{e.definitialdataheur}
\kappa(r)\sum_{j=1}^rA_j\big(\mathbf{v}\cos(\mathbf{k_j}\cdot x)+\mathbf{v'}\cos(\mathbf{k_j'}\cdot x)\big).
\end{equation}
Here $\mathbf{v},\, \mathbf{v'}\in\R^3$ are fixed constant vectors, $\kappa(r)$ is some function such that $\kappa(r)\rightarrow 0$ when $r\rightarrow\infty$,  $A_j\in[0,\infty)$ is a sequence of amplitudes growing geometrically, and $k_j\, k_j'\in\R^3$ are two sequences of vectors whose magnitudes grow at a geometric rate. 
Hence the initial data given by \eqref{e.definitialdataheur} is a superposition of highly oscillating plane waves. 
This data covers the situations studied in \cite{BP08,Y10,CD14}. In all these studies, $\kappa(r)$ and the sequence $A_j$ need to be finely tuned in order to produce a small $\dot B^{-1}_{\infty,\infty}$ norm at initial time but a large one after an arbitrarily short time. 

We now describe the geometric constraints that we put on the vectors $\mathbf{v}$, $\mathbf{v'}$, $\mathbf{k_j}$ and $\mathbf{k_j'}$. There are two obvious conditions. First, in order to satisfy the divergence-free condition, we impose
\begin{equation}\label{e.condini1}\tag{{3rdNI Condition 1}}
    \mathbf v\cdot\mathbf k_j=0=\mathbf v'\cdot \mathbf k_j'.
\end{equation}
Second, in order to have vertical velocity zero initially, we impose 
\begin{equation}\label{e.condini2}\tag{{3rdNI Condition 2}}
    \mathbf{v}\cdot\mathbf{e}_3=\mathbf{v'}\cdot\mathbf{e}_3=0
\end{equation}
where $\mathbf e_3$ is the third vector of the canonical basis of $\R^3$. 
In order to produce norm inflation in $\dot B^{-1}_{\infty,\infty}$ from this superposition of highly oscillating plane waves, one needs {to} produce a non-oscillating function from the interaction of the term oscillating with wavenumber $\mathbf{k_j}$ {and the} term oscillating with wavenumber $\mathbf{k_j'}$. Hence, following \cite{BP08,Y10,CD14}, we impose that there exists a fixed constant vector $\boldsymbol\eta\in\R^3$ such that 
\begin{equation}\label{e.condini3}\tag{{3rdNI Condition 3}}
    \mathbf{k_j}-\mathbf{k_j'}=\boldsymbol\eta.
\end{equation}
The norm inflation mechanism can be seen as a backward energy cascade, producing large-scale, non-oscillating, structures from small-scale, highly oscillating, structures.

We now investigate the conditions needed to get {norm inflation of the third component}. A computation of the second Duhamel iterate leads to the following inflation term 
\begin{equation}\label{e.norminflationterm}
    \kappa(r)^2\sum_{j=1}^r\Big(\int\limits_0^te^{-(|\mathbf{k_j}|^2+|\mathbf{k_j'}|^2)(t-s)}ds\Big)\mathbb P\big(\sin(\boldsymbol\eta\cdot x)((\mathbf{v}\cdot \mathbf{k_i'})\mathbf{v'}-(\mathbf{v'}\cdot \mathbf{k_j})\mathbf{v}\big).
\end{equation}
Notice that the third component of 
\begin{equation}\label{e.argumentdivergence}
\sin(\boldsymbol\eta\cdot x)((\mathbf{v}\cdot \mathbf{k_i'})\mathbf{v'}-(\mathbf{v'}\cdot \mathbf{k_j})\mathbf{v})
\end{equation}
is zero. 
Hence, in order to get {norm inflation on the third component}, one needs the quantity in \eqref{e.argumentdivergence} to have a non-zero divergence, which will impose further constraints on $\mathbf{k_i},\ \mathbf{k_i'},\ \mathbf{v},\ \mathbf{v'}$ and $\boldsymbol\eta$. This is in stark contrast with previous studies \cite{BP08,Y10,CD14}, where the quantity in \eqref{e.argumentdivergence} is divergence-free and hence the norm inflation term {remains in the span of $\mathbf{v}$ and $\mathbf{v'}$}.

Computing the Helmholtz-Leray projection in the norm inflation term \eqref{e.norminflationterm} we get
\begin{equation}\label{inflationterm}
\begin{split}
   &\mathbb P\big(\sin(\boldsymbol\eta\cdot x)((\mathbf{v}\cdot \mathbf{k_j'})\mathbf{v'}-(\mathbf{v'}\cdot \mathbf{k_j})\mathbf{v}\big) \\
   &=\sin(\boldsymbol\eta\cdot x)\Big((\mathbf{v}\cdot \mathbf{k_j'})\mathbf{v'}-(\mathbf{v'}\cdot \mathbf{k_j})\mathbf{v}-\frac{\boldsymbol\eta}{|\boldsymbol\eta|^2}\big((\mathbf v\cdot \mathbf{k_j'})(\mathbf{v'}\cdot\boldsymbol\eta)-(\mathbf{v'}\cdot \mathbf{k_j})(\mathbf v\cdot\boldsymbol\eta)\big)\Big).
\end{split}
\end{equation}
Therefore, we need 
\begin{equation}\label{e.condini4}\tag{{3rdNI Condition 4}}
    \boldsymbol\eta\cdot \mathbf{e}_3\neq 0
\end{equation}
and 
\begin{equation*}
(\mathbf v\cdot \mathbf{k_j'})(\mathbf{v'}\cdot\boldsymbol\eta)-(\mathbf{v'}\cdot \mathbf{k_j})(\mathbf v\cdot\boldsymbol\eta)\neq 0
\end{equation*}
in order to have norm inflation on the third component of the velocity. Using \eqref{e.condini1} we can rewrite the last condition as
\begin{equation*}
    (\mathbf v\cdot \mathbf{k_j'})(\mathbf{v'}\cdot\mathbf{k_j})\neq 0
\end{equation*}
i.e. 
\begin{equation}\label{e.condini5}\tag{{3rdNI Condition 5}}
\mathbf v\cdot \mathbf{k_j'}\neq 0\quad\mbox{and}\quad \mathbf v'\cdot \mathbf{k_j}\neq 0.
\end{equation}
Notice that conditions \eqref{e.condini1}-\eqref{e.condini5} are necessary but also sufficient to have {norm inflation on the third component}. 
There are many possible choices within the constraints \eqref{e.condini1}-\eqref{e.condini5}. In particular, taking 
$$
\mathbf{v}=(1,-\lambda,0)\quad\mbox{and}\quad\mathbf{k_j}=(\lambda,1,2^{j-1}K)
$$
for $\lambda,\, K\in\Z$, one has a whole family of initial data with $u_{\rm in}^3=0$ that produces {norm inflation on the third component} in $\dot B^{-1}_{\infty,\infty}$ and such that $u_{\rm in}$ is close in $\dot B^{-1}_{\infty,\infty}$ to a 2.75D shear flow initial data defined in \eqref{3D-nu-data}.

\subsection{{Heuristics for strong isotropic symmetry breaking}}
\label{sec.heuriso}
{From the previous subsection, notice that for initial data of the form $$ \kappa(r)\sum_{j=1}^rA_j\big(\mathbf{v}\cos(\mathbf{k_j}\cdot x)+\mathbf{v'}\cos(\mathbf{k_j'}\cdot x)\big)$$ that satisfies \eqref{e.condini1}-\eqref{e.condini5}, the associated inflation term \eqref{inflationterm} vanishes in the direction $\boldsymbol{\eta}$. Here
\begin{equation}\label{freqdifference}
\mathbf{k_j}-\mathbf{k_j'}=\boldsymbol{\eta}\quad\textrm{for}\,\textrm{all}\,j.
\end{equation}
This represents a block in using such initial data for obtaining the isotropic norm inflation \eqref{e.sini} in Theorem \ref{Th-3bis}.

To overcome this, we instead take initial data of the form $$\kappa(r)\sum_{j=1}^{r^3}A_j\big(\mathbf{v_{j}}\cos(\mathbf{k_j}\cdot x)+\mathbf{v_{j}'}\cos(\mathbf{k_j'}\cdot x)\big) $$ 
\begin{equation}\label{freqdifferencevary}
\textrm{with}\quad\mathbf{k_j}-\mathbf{k_j'}=\boldsymbol{\eta_{j}}.
\end{equation}
Here, $\mathbf{v_{j}}$ and $\mathbf{v_{j}'}$ vary with $j$ and crucially the low frequency vector $\boldsymbol{\eta_{j}}$ points in different directions depending on the index $j$. Specifically, we glue higher frequency terms to the initial data in Theorem \ref{Th-3}, such that the added terms $\mathbf{v_{j}}$,$\mathbf{v_{j}'}$ and $\boldsymbol{\eta_{j}}$ point in different directions depending on $j$.

 The initial data we design in Theorem \ref{Th-3bis} can be decomposed into three pieces $u_{\rm in}=u_{\rm in}^{(1)}+u_{\rm in}^{(2)}+u_{\rm in}^{(3)}$ such that
 \begin{itemize}
 \item  each $u_{\rm in}^{(1)}$-$u_{\rm in}^{(3)} $  separately generate an associated Navier-Stokes solution with a norm inflation term, with each of these norm inflation terms being of comparable size in $\dot{B}^{-1}_{\infty,\infty}$,
 \medskip
 \item the norm inflation term associated to $u_{\rm in}$ is the sum of the norm inflation terms associated to $u_{\rm in}^{(1)}$-$u_{\rm in}^{(3)} $. 
 \end{itemize}
 Careful choices of $\mathbf{v_{j}}$, $\mathbf{v_{j}'}$, $\mathbf{k_{j}}$ and $\mathbf{k_{j}'}$ then give that the norm inflation terms associated to $u_{\rm in}^{(1)}$-$u_{\rm in}^{(3)} $ point in linearly independent directions. This,  together with our choices of $\mathbf{v_{j}}$, $\mathbf{v_{j}'}$, $\mathbf{k_{j}}$ and $\mathbf{k_{j}'}$ and the fact that $\dot{B}^{-1}_{\infty,\infty}$ is an $L^{\infty}$-based space, enable us to show that for any fixed unit vector $\mathbf{e}$
 \begin{itemize}
 \item[](i)  the dot product of $\mathbf{e}$ with at least one of the norm inflation terms associated to $u_{\rm in}^{(1)}$-$u_{\rm in}^{(3)} $ has a $\dot{B}^{-1}_{\infty,\infty}$ norm with a large lower bound,
 \medskip
 \item[](ii) the lower bound in (i) also serves as a lower bound for the $\dot{B}^{-1}_{\infty,\infty}$ norm of the dot product of $\mathbf{e}$ with the norm inflation term associated to $u_{\rm in}$.
 \end{itemize}
 These features then imply that $u_{\rm in}$ generates a norm inflation term that has large $\dot{B}^{-1}_{\infty,\infty}$ norm in all unit directions. This in turn leads to the results described in Theorem \ref{Th-3bis}.}

\subsection{Heuristics for symmetry breaking despite favorable pressure gradient}
\label{sec.heurfavor}

In this subsection, we give some heuristics for Theorem \ref{Th-2}.

Let us first explain how we design the initial data. Our objective is to find an initial data that will generate symmetry breaking despite favorable initial pressure (see Footnote \ref{foot.favorable}). The two fractions in $u_{\rm in}$ are there to fulfill the condition $\partial_3 P_{\rm in}\equiv 0$, where $P_{\rm in}$ is defined by \eqref{presideqn}. We also remark that the order in $t$ in \eqref{e.loweru3} is expected because $\partial_3 P_{\rm in}=0$ at $t=0$ formally implies that $\partial_tu_3=0$ at $t=0$. The breaking is not driven by the vertical derivative of the pressure at the initial time, as is the case in Theorem \ref{Th-3} and for the Taylor-Green vortex, see Figure \ref{fig.TG}.

In our proof, the condition $N>N_0$ appears for technical reasons in order to identify the leading order term. Indeed, for $N$ large, the term involving $S_1$ is the dominant term in the right hand side of \eqref{es-u23}. Notice also that the larger the $N$, the closer our data is from the two-dimensional data $(\cos x_2,\cos x_1,0)$ that generates a unique global two-dimensional solution to 3D Navier-Stokes. This has two implications. First, one sees, that $u_{\rm in}$ generates a unique 
 solution to the 3D Navier-Stokes equations on $\T^3\times[0,1]$ for $N$ large. Second, the larger the $N$, the weaker the symmetry breaking effect should be. This observation is consistent with the bound $O(t^2/N^2)$ in \eqref{e.loweru3}.

\begin{remark}[on the condition $N\gg 1$ in Theorem \ref{Th-2}]
Figure \ref{fig.ourdata} shows a simulation of the Navier-Stokes equations with initial data from Theorem \ref{Th-2} taking $N=1.1$. The graph shows that symmetry breaking happens in spite of the fact that $N$ is taken small. Therefore we expect that the result of Theorem \ref{Th-2} remains true for $1<N\leq N_0$. 
\end{remark}

\subsection{Outline of the paper}

{Section \ref{S4} is devoted to the proof of strong symmetry breaking, namely Theorem \ref{Th-3} (see Subsection \ref{S4A}) and to the proof of strong isotropic symmetry breaking, namely Theorem \ref{Th-3bis} (see Subsection \ref{S4Abis}).} Section \ref{ill} addresses the proof of Theorem \ref{Th-2}, i.e. symmetry breaking despite pressure favorable to symmetry preservation. The last part of the paper, Section \ref{S3} is concerned with some applications of the 2.75D shear flows, which are symmetry preserving shear flows. This section contains two types of results. First, in Section \ref{S3-s} we investigate inviscid damping effects for 2.75D shear flow solutions of Euler. Second, in Section \ref{s-2.75limit} we study Onsager supercritical inviscid limits of 2.75D shear flows. Finally in Appendix \ref{S3a}, we give another perspective on the derivation of 2.75D shear flows.

\subsection{Notations and preliminary results}

We begin this  section by introducing relevant notation. 
We denote by $C$ positive numerical constants that may change from one line to the other, and we sometimes write $A\lesssim B$ instead of $A\leq C B.$ Likewise,    $A\sim B$ means that  $C_1 B\leq A\leq C_2 B$ with absolute constants $C_1$, $C_2$.  Throughout 
the paper, $i$-th coordinate $(i=1,2,3)$ of a vector $v$ will be denoted by  $v^i,$ and   horizontal component of $v$ will be denoted by $v^{\rm h}.$ For a real-valued matrix $\mathcal{M}$, $\mathcal{M}^{\rm T}$ represents its transpose, while for two multidimensional real-valued matrices $\mathcal{M}_1, \mathcal{M}_2,$ $\mathcal{M}_1:\mathcal{M}_2$ denotes their standard inner product.
 For $X$ a Banach space, $p\in[1, \infty]$ and $T\in(0,\infty]$, the notation $L^p(0, T; X)$ or $L^p_T(X)$ stands for the set of measurable functions $f: [0, T]\to X$ with $t\mapsto\|f(t)\|_X$ in $L^p(0, T)$, endowed with the norm $\|\cdot\|_{L^p_{T}(X)} :=\|\|\cdot\|_X\|_{L^p(0, T)}.$ We  keep the same notation for functions with several components.     

We recall that the Besov spaces $\dot{B}^{-2\sigma}_{\infty,\infty}$ (with $\sigma>0$) is equipped with the norm
\begin{align*}
\|v(\cdot)\|_{\dot{B}^{-2\sigma}_{\infty,\infty}}=\bigl{\|}  \|s^{\sigma} e^{s\Delta} v(\cdot)\|_{L^\infty}\bigr{\|}_{L^\infty(\R_+)}.
\end{align*}
Note also that one has the embedding 
\begin{equation}\label{e.embed}
L^3(\T^3)
\hookrightarrow \dot B^{-1}_{\infty,\infty}(\T^3)
\end{equation}
for mean-free functions on the torus.\footnote{Let us give a short proof of this embedding. For a mean-free function $v$ in $L^3(\T^3)$, for $s>1$,
\begin{align*}
\Big\|\sum_{\xi\in\Z^3\setminus\{0\}}e^{-s|\xi|^2}e^{ix\cdot\xi}\hat v(\xi)\Big\|_{L^\infty(\T^3)}
= &
\Big\|\sum_{\xi\in\Z^3\setminus\{0\}}s|\xi|^2e^{-s|\xi|^2}e^{ix\cdot\xi}\frac{\hat v(\xi)}{s|\xi|^2}\Big\|_{L^\infty(\T^3)}\\
\leq &
\frac1s\Big(\sum_{\xi\in\Z^3\setminus\{0\}}\frac1{|\xi|^4}\Big)^\frac12\Big(\sum_{\xi\in\Z^3\setminus\{0\}}|\hat v(\xi)|^2\Big)^\frac12\\
\leq &
\frac{C}{s}\|v\|_{L^2(\T^3)}\leq 
\frac{C}{s}\|v\|_{L^3(\T^3)},
\end{align*}
where $C\in(0,\infty)$ is a universal constant. 
Notice that we used that the function $xe^{-x}$ is bounded on $\R$. For $s\in(0,1]$, we rely on the result of Maekawa and Terasawa \cite{maekawa2006} for instance.} 
As is usual, we define the bilinear operator
\begin{align*}
\mathcal{B}(u, v)(t, x):= -\int\limits_0^t e^{(t-\tau)\Delta}\mathcal{P}\bigl(u \cdot\nabla v\bigr)(\tau, x) \,d\tau
\end{align*}
with $\mathcal{P}$ the projection on divergence-free vector fields (the so-called Leray projection). 

We need the following  obvious estimates for the one-dimensional heat kernel  
\begin{equation}\label{e.heat1d}
\mathcal{K}(t, x_3):=\frac{1}{\sqrt{4\pi t }}e^\frac{-|x_3|^2}{4t }, \quad\forall ~(t, x_3)\in \R_+\times \T.
\end{equation}
\begin{lemma}\label{Le-heat}
 Let $g\in {C}^\infty( \T),$ then for any $s\in \R_+,$ one has
\begin{align*}
\|   (\mathcal{K}\star g)(s, \cdot) \|_{L^\infty( \T)}\leq \|g\|_{L^\infty(  \T)}
\end{align*}
and
\begin{align*}
\|  (\mathcal{K}\star g)(s, \cdot) -g(\cdot)\|_{L^\infty( \T)}\leq  s\, \|  g''\|_{L^\infty(  \T)}.
\end{align*}
\end{lemma}

Finally, we state a standard absorbing lemma which is useful for our proofs.
 \begin{lemma}\label{Le-absorb}
 Suppose that $y: [0, T]\to [0, \infty)$ is continuous and satisfies $y(0)=0.$ Furthermore suppose that for all $t\in[0, T],$ $y$ satisfies the following inequality:
 \begin{equation*}
 \sup_{s\in[0, t]} y(s)\leq a\Big(\sup_{s\in[0, t]} y(s)\Big)^2+b\sup_{s\in[0, t]} y(s)+c, 
 \end{equation*}
 with $a, b, c>0$ and $b+2ac<\frac{1}{4}.$  Then we conclude that
 \begin{equation*}
 \sup_{s\in[0,T]} y(s)<2c.
 \end{equation*}
 \end{lemma}
 
 \section{Strong symmetry breaking}
 \label{S4}

\subsection{{Proof of Theorem \ref{Th-3}}}
\label{S4A}
In this section, our objective is to prove Theorem \ref{Th-3}. We  investigate the growth of the vertical  velocity  for the three-dimensional  Navier-Stokes problem \eqref{NS}  supplemented with  initial data  $u_{\rm in}$  that is close in the critical Besov spaces $\dot{B}^{-1}_{\infty, \infty}$ to initial data considered in \eqref{3D-nu-data}. For a heuristic description of the growth mechanism with a focus on how to produce {third component norm inflation} from anisotropic initial data, we refer to \eqref{sec.heur}. 
 We proceed in three steps.

\subsection*{Step 1: choice of the initial data}
Let $\Gamma_1, \Gamma_2: \mathbb{N}\mapsto \R$ be such that 
 \begin{align}\label{e.defGamma}
     \Gamma_1(m):= \sum_{j=1}^m \frac{1}{j}\quad{\rm and} \quad \Gamma_2(m):=  \Gamma_1^{\frac13}(m)\quad {\rm for }~m\in \mathbb{N}.
 \end{align}
  Let $r$ be a large integer  (to be specified later).   We  set    initial data  $u_{\rm in}$ and $\bar u_{\rm in}$ as follows:\footnote{We emphasize that $k_j$ is a scalar. Comparing the data to \eqref{e.definitialdataheur}, we see that here $\kappa(r):=\frac1{\Gamma_2(r)}$, $A_j=\frac{k_j}{\sqrt{j}}$, $\mathbf{k_j}=(1,1,k_j)$ and $\mathbf{k_j'}=(0,-1,k_j+1)$. Note also that $u_{\rm in}$ has a large norm in ${BMO^{-1}}(\R^3)$.}
\begin{align*} 
 u_{\rm in}&{=} \frac{1}{\Gamma_2{(r)}}  \sum_{j=1}^r \frac{k_j} {\sqrt{j}}\big(\mathbf{v} ~  {\cos (  x_1+  x_2 +k_jx_3) } 
          + \mathbf{v'} ~ {\cos( - x_2 +(k_j+1)x_3}  )\big),\\
 \bar u_{\rm in}&{=} \frac{1}{\Gamma_2{(r)}}  \sum_{j=1}^r \frac{k_j} {\sqrt{j}}~ \mathbf{v} ~  {\cos (  x_1+   x_2 +k_jx_3) },
\end{align*}
where $\mathbf{v} =\left(1, -1, 0\right), ~\mathbf{v'}=\left(1, 0, 0\right)$ are vectors and we define the sequence $k_j=2^{3j}T^{-\frac{1}{2}}$ ($j=1, \cdots, r$). The existence time $0<T<1$ is to be determined in terms of $r$.

Obviously, $\bar u_{\rm in}$ has the structure \eqref{3D-nu-data} by taking  
$$
\lambda=1,\quad f(y_1, y_2)=\frac{1}{\Gamma_2{(r)}}  \sum_{j=1}^r \frac{k_j} {\sqrt{j}}~ \mathbf{v} ~  {\cos \big(  y_1 +k_j y_2\big) }, \quad g=0.
$$
Thus the vertical velocity of the corresponding 2.75D shear flow remains identically zero for all positive time. 

Notice  that 
 \begin{align*}
u_{\rm in}-\bar u_{\rm in} &=  \frac{1}{\Gamma_2{(r)}}  \sum_{j=1}^r \frac{k_j} {\sqrt{j}} ~ \mathbf{v'}~ {\cos(-x_2+ ({k}_j +1)x_3) },\\
e^{t\Delta}(u_{\rm in}-\bar u_{\rm in})(x)&=\frac{1}{\Gamma_2{(r)}}  \sum_{j=1}^r \frac{k_j} {\sqrt{j}} ~ \mathbf{v'}~ {\cos(-x_2+ ({k}_j+1) x_3)~e^{-((k_j+1)^2+1) t} }
\end{align*}
and   for appropriate $r,$   we have
\begin{align*}
\| u_{\rm in}-\bar u_{\rm in}\|_{\dot{B}^{-1}_{\infty, \infty}}&\leq \frac{1}{\Gamma_2{(r)}}\sup_{s>0}\Big( \sum_{j=1}^r\frac{k_j} {\sqrt{j}}  s^\frac{1}{2}  e^{-((k_j+1)^2+1) s}\Big) \\
&\lesssim \frac{1}{\Gamma_2{(r)}}\sup_{s>0}\Big( \sum_{j=1}^r {k_j}s^\frac{1}{2} e^{-k_j^2s}\Big)\lesssim \frac{1}{\Gamma_2{(r)}}=\Gamma_1^{-\frac{1}{3}}(r).
\end{align*}
In the above and in what follows, we use that series of the type $\sum_{j\in\N}{k_j}s^\frac{1}{2} e^{-k_j^2s}$ and $\sum_{j\in\N}{k_j^2}s e^{-k_j^2s}$ are uniformly bounded in $s$. This can be easily seen by splitting the sum into $\{j:\ 16^j\frac sT<1\}$ and its complement.
 
\subsection*{Step 2: analysis of  the second approximation}  Now, we analyze the second approximate solution associated with initial data $u_{\rm in}$. 
In order to do that, let us first  recall  $u_1(t, x)=e^{t\Delta} u_{\rm in}$ with
 \begin{multline}
  u_{1}(t, x){=} \frac{1}{\Gamma_2{(r)}}  \sum_{j=1}^r \frac{k_j} {\sqrt{j}} 
\Big{(}\mathbf{v}\,
          { \cos (  x_1+  x_2+k_jx_3) }\,e^{-(k_j^2+2) t} \\
         +\mathbf{v'}  \,{\cos(-x_2+ ({k}_j +1)x_3) }\,e^{-((k_j+1)^2+1) t}
          \Big{)},\label{S4-defu1}
        \end{multline}
and    $ v_2:= \mathcal{B}( u_1,   u_1)$  with  
\begin{align*} 
 \  v_2(t, x) =\frac{1}{\Gamma^2_2(r)}\sum_{i=1}^r\sum_{j=1}^r  \int\limits_0^t e^{(t-\tau)\Delta}~ \mathcal{P} U_{i, j}(\tau, x)\,d\tau,
\end{align*}
where    
\begin{align*}
U_{i, j}(\tau, x):=& \frac{k_i k_j}{\sqrt{i j}} \Big(\mathbf{v}  ~G_{i, j}^+(\tau, x)  
 +\mathbf{v'} ~G_{i, j}^-(\tau, x) \Big) 
\end{align*}
and  
\begin{align*}
     G_{i, j}^+ (\tau, x)&:=  {-}{\frac{1}{2}}\Big(\sin\big(  x_1 + 2 x_2 + (k_j-k_i-1)x_3 \big)+\sin\big( x_1 +(k_j+k_i+1) x_3\big)\Big) \\
     &\qquad  \times e^{-(  (k_j^2+(k_i+1)^2+3)\tau},\\
      G_{i, j}^- (\tau, x)&:= {-}{\frac{1}{2}}\Big(\sin\big( -x_1 -2 x_2 +(k_j-k_i+1)x_3\big)+\sin\big( x_1 +(k_j+k_i+1)x_3\big) \Big)   \\
       &\qquad  \times e^{-(  (k_j+1)^2+k_i^2+3)\tau}.
\end{align*}
We see that 
\begin{align*}
    U_{j, j}(\tau, x)
    &= {\frac{1}{2}}\frac{ k_j^2}{j}  {(\mathbf{v'}- \mathbf{v} )}~\sin\big( x_1 + 2x_2-x_3 \big) ~e^{-(  2k_j^2+2k_j+ 4)\tau} \\
    &\quad{-} {\frac{1}{2}}\frac{ k_j^2}{j} (\mathbf{v}+\mathbf{v'})~\sin (x_1 +(2k_j+1)x_3)  ~ e^{-(  2k_j^2+2k_j+ 4)\tau}\\
    & := U_{j, j}^+(\tau, x)+ U_{j, j}^-(\tau, x).
\end{align*}
So we  can decompose $ v_2$ as $v_2=  v_{2, 1}+ v_{2, 2}+  v_{2, 3},$ where
\begin{equation}\label{S4-defv}
\left\{\begin{aligned}
    &  v_{2, 1}(t, x):=\frac{1}{\Gamma^2_2(r)}\sum_{j=1}^r   \int\limits_0^t e^{(t-\tau)\Delta}~ \mathcal{P} U^+_{j, j}(\tau, x)  \,d\tau,\\
    & v_{2, 2}(t, x):=\frac{1}{\Gamma^2_2(r)}\sum_{j=1}^r    \int\limits_0^t e^{(t-\tau)\Delta}~ \mathcal{P}  U^-_{j, j}(\tau, x)\,d\tau ,\\
    &  v_{2, 3}(t, x):=\frac{1}{\Gamma^2_2(r)}\sum_{j=1}^{r}\sum_{i< j}  \int\limits_0^t e^{(t-\tau)\Delta}~ \mathcal{P}  \big(U_{i, j}(\tau, x)+ U_{j, i}(\tau, x)\big)\,d\tau.
\end{aligned}\right.
\end{equation}
Note that $  v_{2, 1}$ will be the term producing the norm inflation.

\begin{lemma}\label{S4-Le}
We have the following key estimates: 
\begin{align}
    \| v_{2, 1}(t, \cdot)\|_{L^\infty(\T^3)}\lesssim \frac{\Gamma_1(r)}{\Gamma_2^2(r)}=\Gamma_1^\frac13(r), \quad{\rm for}~t>0\label{S4-L0}
       \end{align}
       and for each components of $v_{2, 1}$ on the time interval $[T/320, T]$,
       \begin{align}
 \|v_{2, 1}^1(t, \cdot)\|_{\dot{B}^{-1}_{\infty,\infty}}= \|v_{2, 1}^2(t, \cdot)\|_{\dot{B}^{-1}_{\infty, \infty}}= \|v_{2, 1}^3(t, \cdot)\|_{\dot{B}^{-1}_{\infty, \infty}}\gtrsim  {\Gamma^{\frac13}_1(r)},
\label{S4-L1}
       \end{align}
       \begin{align}
 \|v_{2, 1}^1(t, \cdot)\|_{L^3}= \|v_{2, 1}^2(t, \cdot)\|_{L^3}= \|v_{2, 1}^3(t, \cdot)\|_{L^3}\gtrsim {\Gamma^{\frac13}_1(r)}.
\label{S4-L1'}
       \end{align}
Moreover, for $t>0$
    \begin{align}
     &\|  v_{2, 2}(t, \cdot)\|_{L^\infty(\T^3)}\lesssim \frac{1}{\Gamma_2^2(r)}=\Gamma_1^{-\frac23}(r),\label{S4-L2}\\
     & \|  v_{2, 3}(t, \cdot)\|_{L^\infty(\T^3)}\lesssim \frac{1}{\Gamma_2^2(r)}=\Gamma_1^{-\frac23}(r),\label{S4-L3}\\
      &\|u_1(t, \cdot)\|_{L^\infty(\T^3)}\lesssim \frac{1}{\sqrt{t}~\Gamma_2(r)  }= {\Gamma_1^{-\frac{1}{3}}(r)}\,t^{-\frac{1}{2}} \label{S4-L4}.
\end{align}

\end{lemma}

 \begin{proof}[Proof of Lemma \ref{S4-Le}]
 Firstly, a direct computation gives that
 \begin{align*}
 {\mathcal{P}\left((\mathbf{v'}-\mathbf{v}) \sin(x_1+2x_2-x_3)\right)  
 =  \frac{1}{3}\sin(x_1+2x_2-x_3)\,(-1, 1, 1).}
 \end{align*}
Then, by the  definition  of $v_{2, 1}$ in \eqref{S4-defv}  and above equality,
 \begin{align*}
     v_{2, 1}(t, x)=\frac{{(-1, 1, 1)} }{\Gamma^2_2(r)}\sum_{j=1}^r   \frac{k_j^2}{ {6}j} \int\limits_0^t e^{(t-\tau)\Delta}      \sin (x_1+2x_2-x_3)~  ~ e^{-(  2k_j^2+2k_j+ 4)\tau}\,d\tau.
 \end{align*}
Thus  for $t>0$
 \begin{align*}
 \|v_{2, 1}(t, \cdot)\|_{L^\infty}\lesssim\frac{1}{\Gamma^2_2(r)}\sum_{j=1}^r  \frac{k_j^2}{j} \int\limits_0^t e^{-2  k_j^2\tau}\,d\tau\lesssim \frac{\Gamma_1(r)}{\Gamma^2_2(r)}.
 \end{align*}
The vertical component of $v_{2, 1}$ is given by 
 \begin{align}\label{3rdinflation}
  v_{2, 1}^3(t, x) =\frac{{1}}{\Gamma^2_2(r)}\sum_{j=1}^r   \frac{ k_j^2}{6j } \int\limits_0^t e^{(t-\tau)\Delta} \sin(x_1+2x_2-x_3) ~e^{-( 2  k_j^2+2k_j+ 4)\tau}\,d\tau.
 \end{align}
Using this and that $k_1^2 T=64$, we obtain 
\begin{align}\label{thirdinflationproof}
 \|v_{2, 1}^3(t, \cdot)\|_{\dot{B}^{-1}_{\infty, \infty}}&\gtrsim \frac{1}{\Gamma^2_2(r)}~\sup_{s>0}\Big(\sum_{j=1}^r  \frac{ k_j^2}{j}  s^{\frac{1}{2}}\,\int\limits_0^t e^{-6(t-\tau+s)}\,e^{-4 k_j^2\tau}\,d\tau\Big)\notag\\
  &\gtrsim\frac{1}{\Gamma^2_2(r)}~\sum_{j=1}^r  \frac{ e^{-6t}}{j}  ~\int\limits_0^t k_j^2\,e^{-4  k_j^2\tau}\,d\tau \notag\\
  &\gtrsim\frac{1}{\Gamma^2_2(r)}~\sum_{j=1}^r  \frac{1}{j}  ~ (1-e^{-4k_j^2t }) \notag\\
  &\gtrsim \frac{1}{\Gamma^2_2(r)} \sum_{j=1}^r  \frac{ 1}{j} (1-e^{-1}) \gtrsim \frac{\Gamma_1(r)}{\Gamma^2_2(r)}
 \end{align}
 for $ t\in[T/256,  T]$ with $T <1.$ Moreover, we see that the components of $v_{2, 1}$ are comparable, and due to the fact the embedding \eqref{e.embed}, we get \eqref{S4-L1} and \eqref{S4-L1'} easily. 
  
 Next, let us estimate\footnote{In the computation follows, we   drop the Leray projector since its contribution is harmless.} $v_{2, 2}(t, x)$ and $u_1(t, x)$ for $ t>0.$ We have by the  definition of $v_{22}$ in \eqref{S4-defv} that
 \begin{align*}
 \|v_{2, 2}(t, \cdot)\|_{L^\infty}
 &\lesssim \frac{1}{\Gamma^2_2(r)}\sum_{j=1}^r \frac{k_j^2}{j} 
 \int\limits_0^t e^{-(1+(2k_j+1)^2)(t-\tau)  }  e^{-(2k_j^2+2k_j+4)\tau}\,d\tau ,\\
 &\lesssim \frac{1}{\Gamma^2_2(r)}\sum_{j=1}^r k_j^2 t \,e^{-(2k_j^2+2k_j+4)t}~ \frac{1-e^{-(2k_j^2+2k_j-2) t}}{(2k_j^2+2k_j-2)t}\\
 &\lesssim \frac{1}{\Gamma^2_2(r)}\sum_{j=1}^r k_j^2 t \,e^{-k_j^2 t}\lesssim \frac{1}{\Gamma^2_2(r)},
 \end{align*}
 where we used that $\frac{ (1-e^{-(2k_j^2+2k_j-2) t})}{ (2k_j^2+2k_j-2)t }$ is uniformly bounded for $t>0$. 
Similarly, from \eqref{S4-defu1} we have for $t>0,$
 \begin{align*}
 \|\sqrt{t}~u_1(t, \cdot)\|_{L^\infty}&\lesssim  \frac{1}{\Gamma_2(r)} \sum_{j=1}^r    {(k_j^2t)^\frac{1}{2}}  e^{-k_j^2t}\lesssim \frac{1}{\Gamma_2(r)} .
 \end{align*}
Thus, we have shown \eqref{S4-L2} and \eqref{S4-L4}.

Finally, using $\frac{k_j}{2}\leq   {k_j-k_i}-1$ for $i<j$ and $\sum_{i<j} k_i\leq \frac{k_j}{4}$, it is easy to see that
\begin{align*}
-2(k_j-k_i)^2=-2k_i^2-2(k_j-2k_i)k_j\leq -(2k_i^2+k_j^2).
\end{align*}
Therefore we have  for $v_{2, 3}$
 \begin{align*}
 \|v_{2, 3}(t, \cdot)\|_{L^\infty}&\lesssim   \frac{1}{\Gamma_2^2(r)} \sum_{j=1}^r  \sum_{i<j} \frac{k_ik_j}{\sqrt{ij}}\\
 &\qquad\times \int\limits_0^t \left(e^{-  (k_j-k_i-1)^2)(t-\tau)}+ e^{-(k_j+k_i+1)^2(t-\tau)}\right)~e^{-(k_j^2+k_i^2)\tau}\,d\tau\\
 &\lesssim \frac{1}{\Gamma_2^2(r)} \sum_{j=1}^r    k_j^2 \int\limits_0^t e^{-\frac{1}{4}k_j^2(t-\tau)} ~e^{- k_j^2\tau}\,d\tau\\
  &\lesssim \frac{1}{\Gamma_2^2(r)} \sum_{j=1}^r    k_j^2t\, e^{-\frac{1}{4}k_j^2t} \frac{1-e^{-\frac{3}{4}k_j^2t}}{k_j^2t}\lesssim \frac{1}{\Gamma_2^2(r)}.\qedhere
 \end{align*}
\end{proof}

\subsection*{Step 3: error analysis}
We will show that for appropriately chosen $0<T<1$, there exists a solution $u$ on $[0,1]\times\T^3$. We will then analyze the remainder term $w$ between $u$ and the second iterate. Showing the existence of $u$ is equivalent to finding $w$ satisfying the integral equation 
\begin{equation}\label{eq-v1}
w= F_1+ F_2+F_3
\end{equation}
with 
\begin{align*}
F_1:=& \cB(w,   u_1 )+\cB( u_1, w)+ \cB(w, v_2)+ \cB(v_2, w),\\
F_2:=& \cB(w, w),\\
F_3:=&  \cB(u_1, v_2)+ \cB(v_2, u_1)+ \cB(v_2, v_2).
\end{align*}
Then $u$ is given by $u=u_1+v_2+w$. 
From Lemma \ref{S4-Le}, we have for $v_2$ that
\begin{align}
\|v_2(t, \cdot)\|_{L^\infty}&\lesssim \|v_{2, 1}(t, \cdot)\|_{L^\infty}+\|v_{2, 2}(t, \cdot)\|_{L^\infty}+\|v_{2, 2}(t, \cdot)\|_{L^\infty}\notag\\
&\lesssim \Gamma^{\frac13}_1(r).\label{e.estv2Linfty}
\end{align}
From  \eqref{S4-L4}, 
\begin{align}\label{S4-es-u1}
\|u_1(t,\cdot)\|_{L^\infty}\lesssim \Gamma_1^{-\frac{1}{3}}(r)~T^{-\frac{1}{2}}\lesssim \Gamma_1^{\frac{1}{6}}(r).\end{align}
By the $L^\infty$ bilinear estimate and estimates \eqref{e.estv2Linfty}-\eqref{S4-es-u1}, we have for $0<t\leq T<1$,
\begin{align}
\|\cB(A,u_1)(t, \cdot)\|_{L^\infty}
&\lesssim {\Gamma_1^{-\frac{1}{3}}(r)}  \int\limits_0^t (t-\tau)^{-\frac{1}{2}}  \tau^{-\frac{1}{2}}\,d\tau\sup_{t\in[0, T]}\|A(t, \cdot)\|_{L^\infty}\notag\\
&\lesssim  {\Gamma_1^{-\frac{1}{3}}(r)}  \int\limits_0^1 (1-s)^{-\frac{1}{2}}  s^{-\frac{1}{2}}\,ds \sup_{t\in[0, T]}\|A(t, \cdot)\|_{L^\infty}\notag\\
&\lesssim {\Gamma_1^{-\frac{1}{3}}(r)} \sup_{t\in[0, T]}\|A(t, \cdot)\|_{L^\infty},\label{e.esta1}
\end{align}
\begin{align} 
\|\cB(A,v_2)(t, \cdot)\|_{L^\infty}
&\lesssim {\Gamma_1^{\frac13}(r)}  \int\limits_0^t (t-\tau)^{-\frac{1}{2}}   \,d\tau\sup_{t\in[0, T]}\|A(t, \cdot)\|_{L^\infty} \notag\\
&\lesssim {\Gamma_1^{\frac13}(r)}  \sqrt{T}  \sup_{t\in[0, T]}\|A(t, \cdot)\|_{L^\infty},\label{e.esta2}
\end{align} 
\begin{align}
\|\cB(u_1,v_2 )(t, \cdot)\|_{L^\infty}&\lesssim \Gamma_1^{-\frac{1}{3}}(r)   {\Gamma_1^{\frac13}(r)}  \int\limits_0^t (t-\tau)^{-\frac{1}{2}}  \tau^{-\frac{1}{2}} \,d\tau\lesssim 1,\label{e.esta3}
\end{align}
\begin{align}
\|\cB(v_2,   v_2 )(t, \cdot)\|_{L^\infty}&\lesssim   {\Gamma_1^{\frac13}(r)} {\Gamma_1^{\frac13}(r)}  \int\limits_0^t (t-\tau)^{-\frac{1}{2}}   \,d\tau\lesssim \Gamma_1^{\frac23}(r)\sqrt{T}\label{e.esta4}
\end{align}
and
\begin{align}
\|\cB(A,B)(t, \cdot)\|_{L^\infty}&\lesssim\sqrt{T}\sup_{t\in[0, T]}\|A(t, \cdot)\|_{L^\infty}\sup_{t\in[0, T]}\|B(t, \cdot)\|_{L^\infty}.\label{e.esta5}
\end{align}
We take $\delta =\Gamma_1^{-\frac{1}{3}}(r)$ and $T=\Gamma_1^{-1}(r)$. Notice that $T<\delta$ and 
 $\sqrt{T}\big(1 + {\Gamma_1^{\frac23}(r)}\sqrt{T}\big)\lesssim \Gamma_1^{-\frac{1}{3}}(r)\ll 1$ for large $r$. Using this and \eqref{e.esta1}-\eqref{e.esta5}, we can apply \cite[Lemma A.1]{Ga03}. This gives the existence of $w\in C([0,T]\times\R^3)$. We also infer that
\begin{align*}
&\sup_{t\in[0, T]}\|w(t, \cdot)\|_{L^\infty}\notag\\
&\lesssim~ \sup_{t\in[0, T]}\big(\|\cB(w,   u_1 )\|_{L^\infty}+\|\cB(w, v_2)\|_{L^\infty}+\|\cB(w, w)\|_{L^\infty}+\|\cB(u_1, v_2)\|_{L^\infty}\notag\\
&~~+\| \cB(v_2, v_2)\|_{L^\infty}\big)\\
&\lesssim~\Big(\sup_{t\in[0, T]}\|w(t, \cdot)\|_{L^\infty}\Big)^2\sqrt{T}+ \left({\Gamma_1^{-\frac{1}{3}}(r)} + {\Gamma_1^{\frac13}(r)}  \sqrt{T}\right)\sup_{t\in[0, T]}\|w(t, \cdot)\|_{L^\infty}\\
&~~+\left(1 + {\Gamma_1^{\frac23}(r)}\sqrt{T} \right).
\end{align*}
The choice of $T$ made above allows us to apply an absorbing argument (see Lemma \ref{Le-absorb}). Hence we have the following a priori bound for sufficiently large $r$
\begin{equation}\label{S4-es-v}
 \sup_{t\in[0, T]}\|w(t, \cdot)\|_{L^\infty}\lesssim \Gamma_1^{\frac{1}{6}}(r).
\end{equation} 
We now prove the main theorem.  Thanks to Lemma \ref{S4-Le} and \eqref{S4-es-v}, \eqref{S4-es-u1}, we conclude that  for $t\in[T/256, T]$ and large enough $r$  
\begin{align*}
\|u^3(t, \cdot)\|_{ \dot{B}^{-1}_{\infty,\infty}}&=\|u^3_1+v_2^3+w^3\|_{ \dot{B}^{-1}_{\infty, \infty}}\\
&\geq \|v_2^3(t, \cdot)\|_{ \dot{B}^{-1}_{\infty,  \infty}}-\|u_1(t, \cdot)\|_{ \dot{B}^{-1}_{\infty,  \infty}}-\|w(t, \cdot)\|_{ \dot{B}^{-1}_{\infty,  \infty}}\\
&\geq\|v_{2, 1}^3(t, \cdot)\|_{ \dot{B}^{-1}_{\infty,  \infty}}-\|v_{2, 2}(t, \cdot)\|_{ L^\infty}-\|v_{2, 3}(t, \cdot)\|_{ L^\infty}-\|u_1(t, \cdot)\|_{ L^\infty}\notag\\
&\qquad-\|w(t, \cdot)\|_{ L^\infty}\\
&\gtrsim \Gamma_1^\frac{1}{3}(r)-\Gamma_1^{-\frac{2}{3}}(r) -\Gamma_1^{ \frac{1}{6}}(r)\gtrsim \Gamma_1^\frac{1}{3}(r)=\frac{1}{\delta},
\end{align*}
where we used the embedding \eqref{e.embed}. 
Finally, the results stated in Theorem \ref{Th-3} follow from the fact that $u=u_1+v_{2, 1}+v_{2, 2}+v_{2, 3}+w$ and using that
\begin{align*}
\|u_1(t, \cdot)\|_{L^\infty}+\|v_{2,2}(t, \cdot)\|_{L^\infty}+\|v_{2,3}(t, \cdot)\|_{L^\infty}+\|w(t, \cdot)\|_{L^\infty} \ll \Gamma_1^{ \frac{1}{3}}(r),
\end{align*}
we obtain \eqref{S4-eq} from \eqref{S4-L1'}.
 
This completes the proof of Theorem \ref{Th-3}.

\subsection{{Proof of Theorem \ref{Th-3bis}}}
\label{S4Abis}

{
The outcome of the previous proof is that the data \eqref{3D-nu-data} is well-designed to show the norm inflation on the third component. This data will serve as a first building block for constructing the initial data for Theorem \ref{Th-3bis}. Two other blocks will be added in order to prove {Isotropic Norm Inflation} as stated in \eqref{e.sini}. The objective of this construction is to rule out the possibility of compensations between different components.

\subsection*{Step 1: choice of the initial data}
Let $\Gamma_1$ and $\Gamma_2$ be defined as in \eqref{e.defGamma}. 
  Let $r$ be a large integer  (to be specified later).   We  set  initial data  $u_{\rm in}$ as follows:
\begin{align}\label{3D-nu-data-bis}
u_{\rm in}&{=} \frac{1}{\Gamma_2{(r)}}  \sum_{j=1}^{r^3} \frac{k_j} {\sqrt{j}}\big(\mathbf{v_j}\cos (\mathbf{k_j}\cdot x)  
          + \mathbf{v_j'}\cos(\mathbf{k_j'}\cdot x)\big),
\end{align}
where we define the sequence $k_j=2^{3j}T^{-\frac{1}{2}}$ ($j=1, \cdots, r^3$)\footnote{As above, the existence time $0<T<1$ is to be determined in terms of $r$.} and where $\mathbf{v_j},\ \mathbf{v'_j},\ \mathbf{k_j}, \mathbf{k_j'}$ are vectors which (contrary to the construction in  Theorem \ref{Th-3}) depend on $j$ in the following way:
\begin{equation*}
\mathbf{v_j}=\left\{\begin{array}{rl}(1,-1,0),&1\leq j\leq r,\\(1,0,0),&r+1\leq j\leq r^2,\\(1,1,0),&r^2+1\leq j\leq r^3,\end{array}\right.,\quad \mathbf{v_j'}=\left\{\begin{array}{rl}(1,0,0),&1\leq j\leq r,\\(1,1,0),&r+1\leq j\leq r^2,\\(0,1,0),&r^2+1\leq j\leq r^3,\end{array}\right.
\end{equation*}
and 
\begin{equation*}
\mathbf{k_j}=\left\{\begin{array}{rl}(1,1,k_j),&1\leq j\leq r,\\(0,1,k_j),&r+1\leq j\leq r^2,\\(1,-1,k_j),&r^2+1\leq j\leq r^3,\end{array}\right.,\quad \mathbf{k_j'}=\left\{\begin{array}{rl}(0,-1,k_j+1),&1\leq j\leq r,\\(-1,1,k_j),&r+1\leq j\leq r^2,\\(1,0,k_j),&r^2+1\leq j\leq r^3.\end{array}\right.
\end{equation*}
Notice that we have the following relations
\begin{equation*}
\mathbf{v_j}\cdot\mathbf{k_j}=0=\mathbf{v_j'}\cdot\mathbf{k_j'},\quad \mathbf{v_j}\cdot\mathbf{e_3}=0=\mathbf{v_j'}\cdot\mathbf{e_3},
\end{equation*}
which {guarantee} that the data is incompressible and has vanishing third component. 
Moreover, we have a {low frequency vector}
\begin{equation}\label{e.low-wave-number}
\boldsymbol{\eta_j}:=\mathbf{k_j}-\mathbf{k_j'}=\left\{\begin{array}{rl}(1,2,-1),&1\leq j\leq r,\\(1,0,0),&r+1\leq j\leq r^2,\\(0,-1,0),&r^2+1\leq j\leq r^3\end{array}\right.
\end{equation}
{that varies according to $j$.}
{This is key} to the {isotropic} norm inflation mechanism. Notice that 
\begin{equation*}
\mathbf{v_j}\cdot\mathbf{k_j'}=\left\{\begin{array}{rl}1,&1\leq j\leq r,\\-1,&r+1\leq j\leq r^2,\\1,&r^2+1\leq j\leq r^3,\end{array}\right.,\quad \mathbf{v_j}\cdot\boldsymbol{\eta_j}=\left\{\begin{array}{rl}-1,&1\leq j\leq r,\\1,&r+1\leq j\leq r^2,\\-1,&r^2+1\leq j\leq r^3.\end{array}\right.
\end{equation*}

\subsection*{Step 2: analysis of  the second approximation}  

As above, we consider the second Duhamel iterate from which the norm inflation comes 
$$ v_2:= \mathcal{B}(u_1 , u_1),$$
where $u_1:=e^{t\Delta} u_{\rm in}$ is the first Duhamel iterate.    
We identify the inflation term $v_{2,1}$ by decomposing as above, cf. \eqref{S4-defv}: $v_2=  v_{2, 1}+ v_{2, 2}+  v_{2, 3},$ where
\begin{equation}\label{S4-defv-bis}
\left\{\begin{aligned}
    &  v_{2, 1}(t, x):=\frac{1}{\Gamma^2_2(r)}\sum_{j=1}^{r^3}   \int\limits_0^t e^{(t-\tau)\Delta}~ \mathcal{P} U^+_{j, j}(\tau, x)  \,d\tau,\\
    & v_{2, 2}(t, x):=\frac{1}{\Gamma^2_2(r)}\sum_{j=1}^{r^3}    \int\limits_0^t e^{(t-\tau)\Delta}~ \mathcal{P}  U^-_{j, j}(\tau, x)\,d\tau ,\\
    &  v_{2, 3}(t, x):=\frac{1}{\Gamma^2_2(r)}\sum_{j=1}^{r^3}\sum_{i< j}  \int\limits_0^t e^{(t-\tau)\Delta}~ \mathcal{P}  \big(U_{i, j}(\tau, x)+ U_{j, i}(\tau, x)\big)\,d\tau.
\end{aligned}\right.
\end{equation}
{Here,
\begin{align*}
U_{i, j}(\tau, x):=&-\frac{k_ik_j}{2\sqrt{ij}}\mathbf{v_j}\Big((\mathbf{v_i}\cdot\mathbf{k_j})\Big(\sin((\mathbf{k_j}+\mathbf{k_i})\cdot x)+\sin((\mathbf{k_j}-\mathbf{k_i})\cdot x)\Big)e^{-\tau(|\mathbf{k_j}|^2+|\mathbf{k_i}|^2)}\\ 
&+(\mathbf{v_{i}'}\cdot\mathbf{k_j})\Big(\sin((\mathbf{k_j}+\mathbf{{k}_{i}'})\cdot x)+\sin((\mathbf{k_j}-\mathbf{{k}_{i}'})\cdot x)\Big)e^{-\tau(|\mathbf{k_j}|^2+|\mathbf{{k}_{i}'}|^2)} \Big)\\
&-\frac{k_ik_j}{2\sqrt{ij}}\mathbf{{v}_{j}'}\Big((\mathbf{v_i}\cdot\mathbf{{k}_{j}'})\Big(\sin((\mathbf{{k}_{j}'}+\mathbf{k_i})\cdot x)+\sin((\mathbf{{k}_{j}'}-\mathbf{k_i})\cdot x)\Big)e^{-\tau(|\mathbf{{k}_{j}'}|^2+|\mathbf{k_i}|^2)}\\
&+(\mathbf{v_{i}'}\cdot\mathbf{{k}_{j}'})\Big(\sin((\mathbf{{k}_{j}'}+\mathbf{{k}_{i}'})\cdot x)+\sin((\mathbf{{k}_{j}'}-\mathbf{{k}_{i}'})\cdot x)\Big)e^{-\tau(|\mathbf{{k}_{j}'}|^2+|\mathbf{{k}_{i}'}|^2)} \Big)
\end{align*}
and
$$U_{j,j}(\tau,x)=U^+_{j, j}(\tau,x)+U^-_{j, j}(\tau,x), $$ with
\begin{align*}
U^+_{j, j}(\tau,x):=&-\frac{k_{j}^2}{2j}\Big(\mathbf{v_{j}}(\mathbf{v_{j}'}\cdot\mathbf{k_j})\sin((\mathbf{k_j}+\mathbf{{k}_{j}'})\cdot x)e^{-\tau(|\mathbf{k_j}|^2+|\mathbf{{k}_{j}'}|^2)}
+\mathbf{v_{j}'}(\mathbf{v_j}\cdot\mathbf{{k}_{j}'})\sin((\mathbf{{k}_{j}'}+\mathbf{k_j})\cdot x)e^{-\tau(|\mathbf{{k}_{j}'}|^2+|\mathbf{k_j}|^2)}\Big)
\end{align*}}
Using the relation \eqref{e.low-wave-number}, it appears that
\begin{align}\label{v21iso}
     v_{2, 1}(t, x)=&\frac{(-1, 1, 1) }{\Gamma^2_2(r)}\sum_{j=1}^r   \frac{k_j^2}{6j} \int\limits_0^t e^{-6(t-\tau)}e^{-(  2k_j^2+2k_j+ 4)\tau}\sin (x_1+2x_2-x_3)~  \,d\tau\notag\\
&+\frac{(0, -1, 0) }{\Gamma^2_2(r)}\sum_{j=r+1}^{r^2}   \frac{k_j^2}{2j} \int\limits_0^t e^{-(t-\tau)}e^{-(3+2k_j^2)\tau}\sin (x_1)\,d\tau\notag\\
&+\frac{(1,0,0) }{\Gamma^2_2(r)}\sum_{j=r^2+1}^{r^3}   \frac{k_j^2}{2j} \int\limits_0^t e^{-(t-\tau)}e^{-(3+2k_j^2)\tau}\sin (-x_2)\,d\tau.
 \end{align}}

{Essentially the same arguments as in Lemma \ref{S4-Le} yield that for all $t>0$ 
\begin{align}
     &\|  v_{2, 2}(t, \cdot)\|_{L^\infty(\T^3)}\lesssim \frac{1}{\Gamma_2^2(r)}=\Gamma_1^{-\frac23}(r),\label{S4-L2iso}\\
     & \|  v_{2, 3}(t, \cdot)\|_{L^\infty(\T^3)}\lesssim \frac{1}{\Gamma_2^2(r)}=\Gamma_1^{-\frac23}(r),\label{S4-L3iso}\\
      &\|u_1(t, \cdot)\|_{L^\infty(\T^3)}\lesssim \frac{1}{\sqrt{t}~\Gamma_2(r)  }= {\Gamma_1^{-\frac{1}{3}}(r)}\,t^{-\frac{1}{2}} \label{S4-L4iso}.
\end{align}
Let us focus on obtaining a lower bound in $\dot{B}^{-1}_{\infty,\infty}$ of the dot product of $v_{2,1}(t,\cdot)$ with any unit direction. This is the main difference with respect to the proof of Theorem \ref{Th-3}. We claim that}{ for all $t\in [T/320,T]$, for $r\geq 64$,}
{
\begin{equation}\label{v21inflationiso}
\inf_{\mathbf{e}\in\mathbb{R}^3:|\mathbf{e}|=1}\|v_{2,1}(t,\cdot)\cdot\mathbf{e}\|_{\dot{B}^{-1}_{\infty, \infty}}\gtrsim \Gamma_1^\frac13(r).
\end{equation}}

{To show this, we make use of the structure of the inflation term $v_{2,1}$ in \eqref{v21iso} and we also utilize the following}
{ simple fact from algebra
\begin{equation}\label{algebrafact}
\max\left(\left|\left(\begin{array}{c}\alpha\\ \beta\\ \gamma\end{array}\right)\cdot\left(\begin{array}{c}-1\\1\\1\end{array}\right)\right|,\left|\left(\begin{array}{c}\alpha\\ \beta\\ \gamma\end{array}\right)\cdot\left(\begin{array}{c}0\\-1\\0\end{array}\right)\right|,\left|\left(\begin{array}{c}\alpha\\ \beta\\ \gamma\end{array}\right)\cdot\left(\begin{array}{c}1\\0\\0\end{array}\right)\right|\right)\geq {\frac1{4\sqrt{2}}}.
\end{equation}}
{According to \eqref{algebrafact}, first suppose that the unit vector $\mathbf{e}=(\alpha,\beta,\gamma)$ satisfies \begin{equation}\label{case1algebra}
\left|\left(\begin{array}{c}\alpha\\ \beta\\ \gamma\end{array}\right)\cdot\left(\begin{array}{c}-1\\1\\1\end{array}\right)\right|\geq {\frac1{4\sqrt{2}}}. 
\end{equation}
Using this, the form of $v_{2,1}$ in \eqref{v21iso}  and the same arguments as in \eqref{3rdinflation}-\eqref{thirdinflationproof}, we obtain that for $t\in [T/320,T]$
\begin{align*}\sup_{s>0}s^{\frac{1}{2}}|e^{s\Delta}v_{2,1}(t,0,0,\tfrac{\pi}{2})\cdot\mathbf{e}|\gtrsim  \frac{\sup_{s>0} s^{\frac{1}{2}}e^{-6s}}{\Gamma^{2}_{2}(r)}\sum_{j=1}^{r}\frac{k_{j}^2}{j}\int\limits_{0}^{t}e^{-6(t-\tau)}e^{-\tau(2k_j^2+2k_j+4)} d\tau\gtrsim \Gamma^{\frac{1}{3}}_{1}(r).
\end{align*}
Hence, in the first case \eqref{case1algebra} we get that for all $t\in [T/320,T]$ and $r\geq 64$
$$\|v_{2,1}(t,\cdot)\cdot\mathbf{e}\|_{\dot{B}^{-1}_{\infty, \infty}}\gtrsim \Gamma_1^\frac13(r). $$}
{For the second case according to \eqref{algebrafact}, suppose that the unit vector $\mathbf{e}=(\alpha,\beta,\gamma)$ satisfies \begin{equation}\label{case2algebra}
\left|\left(\begin{array}{c}\alpha\\ \beta\\ \gamma\end{array}\right)\cdot\left(\begin{array}{c}0\\-1\\0\end{array}\right)\right|\geq {\frac1{4\sqrt{2}}}. 
\end{equation}
From this, the form of $v_{2,1}$ in \eqref{v21iso}  and similar arguments as in the first case, we obtain that for $t\in [T/320,T]$ and $r\geq 64$
\begin{align*}\sup_{s>0}s^{\frac{1}{2}}|e^{s\Delta}v_{2,1}(t,\tfrac{\pi}{2},0,\tfrac{\pi}{2})\cdot\mathbf{e}|&\gtrsim  \frac{\sup_{s>0} s^{\frac{1}{2}}e^{-s}}{\Gamma^{2}_{2}(r)}\sum_{j=r+1}^{r^2}\frac{k_{j}^2}{j}\int\limits_{0}^{t}e^{-(t-\tau)}e^{-\tau(2k_j^2+3)} d\tau\\
&\gtrsim\frac{1}{{\Gamma^{2}_{2}(r)}}\sum_{j=r+1}^{r^2}\frac{1}{j}\gtrsim\frac{\Gamma_{1}(r)}{\Gamma^{2}_{2}(r)}\gtrsim \Gamma^{\frac{1}{3}}_{1}(r).
\end{align*}
Hence, in the second case \eqref{case2algebra} we get that for all $t\in [T/320,T]$ and $r\geq 64$
$$\|v_{2,1}(t,\cdot)\cdot\mathbf{e}\|_{\dot{B}^{-1}_{\infty, \infty}}\gtrsim \Gamma_1^\frac13(r). $$}
{For the third and final case according to \eqref{algebrafact}, suppose that the unit vector $\mathbf{e}=(\alpha,\beta,\gamma)$ satisfies \begin{equation}\label{case3algebra}
\left|\left(\begin{array}{c}\alpha\\ \beta\\ \gamma\end{array}\right)\cdot\left(\begin{array}{c}1\\0\\0\end{array}\right)\right|\geq {\frac1{4\sqrt{2}}}. 
\end{equation}
From this, the form of $v_{2,1}$ in \eqref{v21iso}  and similar arguments as in the previous cases, we obtain that for $t\in [T/320,T]$ and $r\geq 64$
\begin{align*}\sup_{s>0}s^{\frac{1}{2}}|e^{s\Delta}v_{2,1}(t,0,\tfrac{\pi}{2},0)\cdot\mathbf{e}|&\gtrsim  \frac{\sup_{s>0} s^{\frac{1}{2}}e^{-s}}{\Gamma^{2}_{2}(r)}\sum_{j=r^2+1}^{r^3}\frac{k_{j}^2}{j}\int\limits_{0}^{t}e^{-(t-\tau)}e^{-\tau(2k_j^2+3)} d\tau\\
&\gtrsim\frac{1}{{\Gamma^{2}_{2}(r)}}\sum_{j=r^2+1}^{r^3}\frac{1}{j}\gtrsim\frac{\Gamma_{1}(r)}{\Gamma^{2}_{2}(r)}\gtrsim \Gamma^{\frac{1}{3}}_{1}(r).
\end{align*}
Hence, in the third case \eqref{case3algebra} we get that for all $t\in [T/320,T]$ and $r\geq 64$
$$\|v_{2,1}(t,\cdot)\cdot\mathbf{e}\|_{\dot{B}^{-1}_{\infty, \infty}}\gtrsim \Gamma_1^\frac13(r). $$
Combing these three cases, we see that we have established \eqref{v21inflationiso}.}

 {Using \eqref{v21inflationiso} with \eqref{S4-L2iso}-\eqref{S4-L4iso}, we see that} {the final error analysis is carried out as in \textbf{Step 3} of {Theorem \ref{Th-3}} above, chosing $\delta =\Gamma_1^{-\frac{1}{3}}(r)$ and $T=\Gamma_1^{-1}(r)$. This concludes the proof of Theorem \ref{Th-3bis}.}

\section{Symmetry breaking in the presence of favorable pressure}
\label{ill}

The objective of this section is to prove Theorem \ref{Th-2}.  
In the following,  we construct an initial data $(u_{\rm in}^{\rm h}, 0)$ satisfying condition \eqref{Structure-initialdata} and such that the condition $u^3\equiv 0$ is instantly broken for the Navier-Stokes problem \eqref{NS}. 
For further insights about the heuristics behind our construction, we refer to Section \ref{sec.heurfavor}. 

In this section, we use the data introduced in Theorem \ref{Th-2}:
\begin{equation*}
 u_{\rm in}{=}   \left(  
            \cos x_2\, \frac{N}{N+\sin x_3},\,
   \cos x_1\,  \frac{N+\sin x_3}{N},\,
         0\right).
\end{equation*}
First, let us explain why a unique solution $u$ exists on $[0,1]\times\T^3$ for $N$ sufficiently large. Let 
$$
u_{\rm in}^{2D}=(\cos x_2,\cos x_1,0)
$$
and let $u^{2D}\in L^\infty((0,1)\times\T^3)$ be the two-dimensional global solution. Then,
\begin{equation}\label{e.smallU}
u_{\rm in}-u_{\rm in}^{2D}=\left(  
            -\cos x_2\, \frac{\sin x_3}{N+\sin x_3},\,
   \cos x_1\,  \frac{\sin x_3}{N},\,
         0\right)
\end{equation}
and we see finding $u$ is equivalent to finding $U$ on $[0,1]\times\T^3$ satsifying
\begin{equation}\label{e.dumamelU}
U=e^{t\Delta}(u_{\rm in}-u_{\rm in}^{2D})+\cB(U,u^{2D})+\cB(u^{2D},U)+\cB(U,U).
\end{equation}
Using the previously discussed $L^\infty$-bilinear estimates and \eqref{e.smallU}, we see that for sufficiently large $N$ we can apply \cite[Lemma A.1]{Ga03} on successive time intervals to get existence of $U\in L^\infty((0,1)\times\T^3)$ satisfying \eqref{e.dumamelU}.

We rewrite the Navier-Stokes equations \eqref{NS} as
\begin{equation*}
u=e^{t\Delta} u_{{\rm in}}+\mathcal{B}(u, u).
\end{equation*}
Now, we define the first and the second approximate solutions 
in the following natural way: let $u_1=e^{t\Delta} u_{\rm in}$ and 
\begin{align*}
u_2:= e^{t\Delta} u_{\rm in}+v_2\quad{\rm  with }\quad v_2:=\mathcal{B}(u_{ 1}, u_{ 1}).
\end{align*}
We denote the difference between $u$ and the second approximation $u_2$ by $w$. Then $w$  satisfies the integral equation \eqref{eq-v1}.

\subsection*{Step 1: analysis of  $v_2$}
We show that for the initial data $u_{\rm in}$ the third component of the first approximate solution $u_2^3$ has a non-zero $\dot{B}^{0}_{\infty, \infty}$ norm for a short time interval. Notice that
\begin{equation}\label{first-ite} 
 u_1(t, x){=}  \left(  \begin{array}{c}
           e^{-t} \cos x_2\, e^{t\partial_3^2}(\frac{N}{N+\sin x_3})\\
   e^{-t}\cos x_1\,\frac{N+e^{-t}\sin x_3}{N}\\
         0
          \end{array}
          \right)\cdotp
\end{equation}
Recalling the definition
\begin{align*}
v_2(t, x)= -\int\limits_0^t e^{(t-\tau)\Delta}\mathcal{P}(u_1\cdot\nabla u_1)(\tau, x)\,d\tau,
\end{align*}
we have  
\begin{align}\label{ap-u3}
 v_2(t, x)&= \int\limits_0^t e^{(t-\tau)\Delta}\mathcal{P} \left(  \begin{array}{c}
           e^{-2\tau} \cos x_1\sin x_2 \, F(\tau, x_3)\\
   e^{-2\tau}\sin x_1 \cos x_2 \, F(\tau, x_3)\\
         0
          \end{array}
          \right)\,d\tau\notag\\
    &=- 2\int\limits_0^t e^{-2\tau} e^{(t-\tau)\Delta}\mathcal{P} \left(  \begin{array}{c}
          0\\
  0\\
         \sin x_1\sin x_2 \, \partial_3 F(\tau, x_3)
          \end{array}
          \right)\,d\tau\notag\\
    &=- \int\limits_0^t   e^{-2\tau}e^{(t-\tau)\Delta} (-\Delta)^{-1}\left(  \begin{array}{c}
          \cos x_1 \sin x_2 \,\partial_{3}^2 F(\tau, x_3)\\
          \sin x_1\cos _2 \,\partial_{3}^2 F(\tau, x_3)\\
      2\sin x_1 \sin x_2 \partial_3 F(\tau, x_3)\\
          \end{array}
          \right)\,d\tau, 
\end{align}
where $F(\tau, x_3):=\frac{N+e^{-\tau}\sin x_3}{N}\, e^{\tau\partial_3^2}\Bigl(\frac{N}{N+\sin x_3}\Bigr).$

Since we are not able to write an explicit formula for $ F$,  we need to determine the main contributions of $F$ while keeping control of the remainder parts. Unlike the case of \cite{Wa15},  there is no way to use the Taylor  series $e^{\tau\Delta}=\sum_{j\in \mathbb{N}}  \frac{(\tau\Delta)^j}{j!}$ to single out the main parts of $F$. Indeed at each order of $e^{\tau\partial_3^2}\frac{\cos x_3}{(N+\sin x_3)^2}$ there is a remainder 
term $\frac{\cos x_3}{(N+\sin x_3)^2}$ and thus we are not able to control the tail, even for a short time. Therefore, our idea is to first write a Taylor expansion for $\frac{N}{N+\sin x_3}$ and then compute the associated heat flows. 
 
Since $u^3_2= v_2^3$ we need to consider
\begin{align*}
\partial_3 F(\tau, x_3)
&=    e^{-\tau}  \Bigl\{N(1-e^{\tau})  \,e^{\tau\partial_3^2}\Bigl(\frac{\cos x_3}{(N+\sin x_3)^2}\Bigr)+\cos x_3\, (e^{\tau\partial_3^2}-1) \Bigl(\frac{1}{N+\sin x_3}\Bigr)\\
&\qquad\qquad-(N+\sin x_3)  (e^{\tau\partial_3^2}-1)\Bigl(\frac{\cos x_3}{(N+\sin x_3)^2}\Bigr)\Bigr\}.
\end{align*}
It is clear that due to the structure of the initial data one has that $\partial_3 F(\tau, x_3)|_{\tau=0}=0$ and $\partial_\tau \partial_3 F(\tau, x_3)|_{\tau=0}\sim \frac{1}{N^2}$. 
Thus, for a short time 
$\partial_3 F(\tau, x_3)\sim \frac{\tau}{N^2}.$

Using the fact that 
\begin{align*}
\frac{1}{N+\sin x_3}=\frac{1}{N}\,\frac{1}{1+\frac{\sin x_3}{N}}=\frac{1}{N}\sum_{j\in \mathbb{N}} (- \frac{\sin x_3}{N})^j,
\end{align*}
we write 
\begin{align*}
\frac{1}{N+\sin x_3}= \frac{1}{N}-\frac{\sin x_3}{N^2}+ R_1(x_3)\quad\mbox{with}\quad R_1(x_3):= \frac{\sin^2 x_3}{(N+\sin x_3)N^2}.
\end{align*}
Then, we have
\begin{multline}\label{ap-es1}
(e^{\tau \partial_3^2} -1) \Bigl(\frac{1}{N+\sin x_3}\Bigr)= (e^{\tau \partial_3^2} -1)\Bigl(\frac{1}{N}-\frac{\sin x_3}{N^2}+ R_1(x_3)\Bigr)\\
=\frac{\sin x_3}{N^2} (1-e^{-\tau}) + (e^{\tau \partial_3^2} -1)R_1(x_3).
\end{multline}
Notice that $\frac{d}{d x_3}\frac{1}{N+\sin x_3}=-\frac{\cos x_3}{(N+\sin x_3)^2}$, so one has that
\begin{align*} 
 \frac{\cos x_3}{(N+\sin x_3)^2}= \frac{\cos x_3}{N^2}-\frac{\sin (2x_3)}{N^3}+ R_2(x_3)\quad\mbox{with}\quad R_2:=   \frac{(N+2\sin x_3)\sin x_3\cos x_3}{(N+\sin x_3)^2N^3}. 
\end{align*}
Furthermore,
\begin{multline}\label{ap-es2}
(e^{\tau \partial_3^2} -1) \Bigl(\frac{\cos x_3}{(N+\sin x_3)^2}\Bigr)= (e^{\tau \partial_3^2} -1)\Bigl(\frac{\cos x_3}{N^2}-\frac{\sin (2x_3)}{N^3}+ R_2(x_3)\Bigr)\\
= \frac{\cos x_3}{N^2} (e^{-\tau}-1)  +\frac{\sin (2x_3)}{N^3}(1-e^{-4\tau}) +(e^{\tau \partial_3^2} -1)R_2(x_3).
\end{multline}
For  the first term in the formula of $F$, we note that
\begin{align}\label{ap-es3}
e^{\tau \partial_3^2}\Bigl(\frac{\cos x_3}{(N+\sin x_3)^2}\Bigr)= \frac{\cos x_3}{N^2} e^{-\tau}-\frac{\sin(2 x_3)}{N^3} e^{-4\tau}+ e^{\tau \partial_3^2}R_2(x_3).
\end{align}
Applying \eqref{ap-es1}-\eqref{ap-es3} into the formula of $\partial_3 F$ and then using \eqref{ap-u3}, we get
\begin{align*} 
\Delta u_2^3(t, x)&=   \frac{2}{N^2}\int\limits_0^t e^{(t-\tau)\Delta}   \Bigl\{ \sin x_1\sin x_2\sin (2 x_3) ( e^{-6\tau}- e^{-4\tau})\Bigr\}\,d\tau \\
&\quad- \frac{2 }{N^3}\int\limits_0^t e^{(t-\tau)\Delta}   \Bigl\{ \sin x_1\sin x_2\sin x_3\sin (2 x_3) ( e^{-3\tau}- e^{-7\tau})\Bigr\}\,d\tau\\
&\quad+  \int\limits_0^t e^{(t-\tau)\Delta}    2e^{-3\tau}\sin x_1\sin x_2\Bigl\{N(1-e^{\tau})  e^{\tau \partial_3^2}R_2(x_3)   \\
&  \qquad\quad \quad+ \cos x_3 (e^{\tau \partial_3^2} -1)R_1(x_3) -  (N+\sin x_3) (e^{\tau \partial_3^2}-1)R_2(x_3)\Bigr\}\,d\tau\\
&:= S_1(t, x) +S_2(t, x)+S_3(t, x).
\end{align*}
Now, we are ready to estimate $\Delta u^3_2(t, x)$ in the space $\dot{B}^{-2}_{\infty, \infty}$ for a small time  $t~ (0< t\ll1)$.  
For $s>0,$ we find that 
 \begin{align*}
e^{s\Delta} S_1(t, x)&= \frac{2 }{N^2} \sin x_1\sin x_2\sin (2 x_3)\int\limits_0^t e^{-6(t+s-\tau)}   ( e^{-6\tau}- e^{-4\tau})\,d\tau\\
 &=\frac{2 }{N^2} \sin x_1\sin x_2\sin (2 x_3)(te^{-6t}-\frac{1}{2}(e^{-4t}-e^{-6t}))e^{-6s},
 \end{align*}
 thus  
  \begin{align}
  \begin{split}\label{es-S1}
\|S_1(t, \cdot)\|_{\dot{B}^{-2}_{\infty, \infty}}&=\sup_{s>0} s \|e^{s\Delta}S_1(t, \cdot)\|_{L^\infty(\T^3)}\\
&=\frac{1}{N^2}( e^{-4t}-(2t+1)e^{-6t}) \,\sup_{s>0} s e^{-6s}\geq \frac{1}{3 e} \frac{t^2 }{N^2}.
\end{split}
 \end{align}
To estimate  $e^{s\Delta} S_2,$ let us recall the formula
 $$\sin x_3\sin (2 x_3)=\frac{1}{2}\cos(x_3)-\frac{1}{2}\cos(3x_3),$$
 then 
 \begin{align*}
e^{s\Delta}S_2(t, x)&=  \frac{ 1}{N^3}\sin x_1\sin x_2e^{-2(t+s)}\int\limits_0^t  e^{(t+s-\tau)\partial_3^2}\Bigl( \cos (2 x_3)-\cos x_3\Bigr) ( e^{-\tau}- e^{-5\tau})\,d\tau\\
&= {\frac{1}{N^3}\sin x_1\sin x_2\cos(3x_3) \left(\frac{1}{8}e^{-3t}+\frac{1}{4}e^{-7t}+\frac{1}{8}e^{-11t}\right)e^{-11s}}\\
&\quad -\frac{1}{N^3}\sin x_1\sin x_2\cos x_3\left(t e^{-3t}+\frac{1}{4}e^{-7t}-\frac{1}{4}e^{-3t}\right)e^{-3s}.
 \end{align*}
 Thus, we have
 \begin{align*}
 &\quad\|e^{s\Delta} S_2(t, \cdot)\|_{L^\infty(\T^3)}\\
 &\leq  {\frac{1}{N^3} \left(\frac{1}{8}e^{-3t}+\frac{1}{4}e^{-7t}+\frac{1}{8}e^{-11t}\right)e^{-11s}}+\frac{1}{N^3}\left(t e^{-3t}+\frac{1}{4}e^{-7t}-\frac{1}{4}e^{-3t}\right)e^{-3s}
 \end{align*}
 and
 \begin{multline}\label{es-S2}
\|S_2(t, \cdot)\|_{\dot{B}^{-2}_{\infty, \infty}}= \sup_{s>0} s \|e^{s\Delta}S_2(t, \cdot)\|_{L^\infty(\T^3)}\\
\leq \frac{1}{N^3}\left(e^{-11t}\left(\frac{1}{8}e^{8t}+\frac{1}{4}e^{4t}+\frac{1}{8}\right)+e^{-3t}\left(t+\frac{1}{4}e^{-4t}-\frac{1}{4}\right)\right)\, \sup_{s>0} s  e^{-3s}\leq \frac{4}{3 e}\frac{t^2}{N^3}. 
 \end{multline}
To estimate $S_3$,  we write
\begin{align*}
 e^{s\Delta}S_3(t, x)&=2 e^{-2(t+s)} \sin x_1\sin x_2 \int\limits_0^t    e^{-\tau }\Bigl\{N(1-e^{\tau})  e^{(t+s) \partial_3^2}R_2(x_3)   \\
&\quad  +e^{(t+s-\tau )\partial_3^2} \Bigl(\cos x_3 (e^{\tau \partial_3^2} -1)R_1(x_3)-   (N+\sin x_3) (e^{\tau \partial_3^2}-1)R_2(x_3)\Bigr)\Bigr\}\,d\tau,
\end{align*}
and by Lemma \ref{Le-heat}
\begin{multline*}
 \|e^{s\Delta} S_3(t, \cdot)\|_{L^\infty(\T^3)}
 \leq  2 e^{-2(t+s)} \int\limits_0^t  e^{-\tau }\Bigl\{ ( e^{\tau}-1)  N\|R_2\|_{L^\infty(\T)}\\
 \qquad\quad+ \tau \|R_1''\|_{L^\infty(\T)}+ (N+1)\tau \|R_2''\|_{L^\infty(\T)}\Bigr\}\,d\tau\\
 \leq 2 e^{-2(t+s)}\Bigl((t - 1+e^{-t}) N\|R_2\|_{L^\infty(\T)}+ (1-(t+1)e^{-t})(\|R_1''\|_{L^\infty(\T)}+ (N+1)\|R_2''\|_{L^\infty(\T)})\Bigr\}.
\end{multline*}
It is easy to check that 
\begin{align*}
\|R_2\|_{L^\infty(\T)}\leq  {\frac{16}{N^4}}
\end{align*}
and
\begin{align*}
\|R_1''\|_{L^\infty(\T)}\leq  {\frac{48}{N^3}},\quad \|R_2''\|_{L^\infty(\T)}\leq  {\frac{928}{N^4}}.
\end{align*}
 Thus 
 \begin{align}\label{es-S3}
\|S_3(t, \cdot)\|_{\dot{B}^{-2}_{\infty, \infty}}=\sup_{s>0} s  \|e^{s\Delta} S_3(t, \cdot)\|_{L^\infty(\T^3)} 
\leq  \frac{496}{e}\frac{t^2}{N^3}.
 \end{align} 
Let $N_0:=3000$. Combining  \eqref{es-S1}-\eqref{es-S3}, we have for all $N>N_0$, 
 \begin{align}
 \begin{split}\label{es-u23}
 \|u_2^3(t, \cdot)\|_{L^\infty}&\geq \|\Delta u_2^3(t, \cdot)\|_{\dot{B}^{-2}_{\infty, \infty}}\\
 &\geq \|S_1(t, \cdot)\|_{\dot{B}^{-2}_{\infty, \infty}}-\|S_2(t, \cdot)\|_{\dot{B}^{-2}_{\infty, \infty}}-\|S_3(t, \cdot)\|_{\dot{B}^{-2}_{\infty, \infty}}   \geq  \frac{1}{6 e}\frac{t^2}{N^2}.
 \end{split}
 \end{align}
 
\subsection*{Step 2: further analysis of $v_2$}
Similar to previous computations, one obtains that 
\begin{align}
    \|\partial_3 F(\tau, \cdot)\|_{L^\infty(\T)}+\|\partial_3^2 F(\tau, \cdot)\|_{L^\infty(\T)}\lesssim  \frac{\tau}{N}.\label{F}
\end{align}
Thus,  from \eqref{ap-u3} and \eqref{F} we have
\begin{align}\label{es-v2}
    \|v_2(t, \cdot)\|_{L^\infty(\T^3)}\lesssim  \|\Delta v_2(t, \cdot)\|_{L^\infty(\T^3)}\lesssim    \int\limits_0^t e^{-2\tau } \frac{\tau}{N}\,d\tau\lesssim   \frac{t^2}{N}.
\end{align}
Meanwhile, from \eqref{first-ite} we see that 
\begin{align}\label{es-u1}
    \|u_1(t, \cdot)\|_{L^\infty(\T^3)}\lesssim   e^{-t}.
\end{align}

\subsection*{Step 3: final error estimate}\label{S4.2}
Now we analyze the remaining part of the solution, which we denote by $w$. We use $L^\infty$ bilinear estimates for controlling the error. 
Recall equation \eqref{eq-v1}, 
estimates \eqref{es-v2},  \eqref{es-u1}. Therefore using the equation for the perturbation \eqref{eq-v1} and the estimates \eqref{es-v2},  \eqref{es-u1}, we have for all $0<T<1$,
\begin{align*}
&\sup_{0\leq t\leq T}\|w(t, \cdot)\|_{L^\infty(\T^3)}\\
\leq& ~\sup_{0\leq t\leq T}\big(\|\cB(w,   u_1 )\|_{L^\infty}+\|\cB( u_1, w)\|_{L^\infty}+ \|\cB(w, v_2)\|_{L^\infty}+ \|\cB(v_2, w)\|_{L^\infty}\notag\\
 &~~+\|\cB(w, w)\|_{L^\infty}
\|B(u_1, v_2)\|_{L^\infty}+ \|\cB(v_2, u_1)\|_{L^\infty}+\| \cB(v_2, v_2)\|_{L^\infty}\big)\notag\\
\lesssim&~ \sup_{0\leq t\leq T}\big(\|\cB(w,   u_1 )\|_{L^\infty}+\|\cB(w, v_2)\|_{L^\infty}+\|\cB(w, w)\|_{L^\infty}+\|\cB(u_1, v_2)\|_{L^\infty}\notag\\
&~~+\| \cB(v_2, v_2)\|_{L^\infty}\big)\\
\lesssim& T^\frac{1}{2}\left(1+ \frac{T^2}{N}+\sup_{0\leq t\leq T}\|w(t, \cdot)\|_{L^\infty(\T^3)}\right)\sup_{0\leq t\leq T}\|w(t, \cdot)\|_{L^\infty(\T^3)}+ \left( \frac{T^\frac{5}{2}}{N}+   \frac{T^{\frac{9}{2}}}{N^2}\right).
\end{align*}
By an absorbing argument (see Lemma \ref{Le-absorb}) for $T\ll 1$ and $N\gg 1$, we obtain   
\begin{align}\label{es-v}
\sup_{0\leq t\leq T}\|w(t, \cdot)\|_{L^\infty(\T^3)}\lesssim    \frac{T^\frac{5}{2}}{N}.
\end{align}
Using the fact that  $u^3=u_2^3+w^3$, estimates \eqref{es-u23} and \eqref{es-v} imply that
\begin{align*}
\|u^3(t, \cdot)\|_{ L^\infty(\T^3)}&\geq \|u_2^3(t, \cdot)\|_{L^\infty(\T^3)}-\|w(t, \cdot)\|_{ L^\infty(\T^3)}\\
&\gtrsim   t^2\Bigl(\frac{1}{N^2}-\frac{  t^\frac{1}{2}}{N}\Bigr)\gtrsim   \frac{t^2}{N^2}
\end{align*}
and
 \begin{align*}
\|u^3(t, \cdot)\|_{ L^\infty(\T^3)}&\leq \|u_2^3(t, \cdot)\|_{L^\infty(\T^3)}+\|w(t, \cdot)\|_{ L^\infty(\T^3)}\lesssim     \frac{t^2}{N^2}
 \end{align*}
 for any $0\leq t\leq T\leq  \frac{1}{2N^2}\ll1.$
 
This completes the proof of Theorem \ref{Th-2}.

\section{2.75D shear flows}
\label{S3}

\subsection{2.75D shear flows and nonlinear inviscid damping}\label{S3-s}

Let us now describe the aforementioned example of a 3D Euler solution that weakly converges but does not strongly converges  to the Kolmogorov flow $u^K=(0,\sin x_3,0)$ in $L^2(\T^3)$ as $t\to \infty$. 
In \eqref{3D-Euler-solu}, take  $ f(y_1, y_2)=
\sin y_1$ 
and $ g(y_3)=-\sin y_3$. Then  the following smooth  initial data  
\begin{equation*} 
(\sin(\lambda  x_1+x_2),\,
         -\lambda 
        \sin(\lambda  x_1+x_2)+\sin x_3,\,
         0), \quad x\in\T^3  
\end{equation*}
  gives a 
  global-in-time solution to the 3D Euler equations of the form
 \begin{equation}\label{2.75kolmogorov}
  u^{E}=\left(  \begin{array}{c}
        \sin\bigl(\lambda  x_1+x_2-t\sin x_3\bigr)\\
         -\lambda
         \sin\bigl(\lambda  x_1+x_2- t\sin x_3\bigr)+\sin x_3\\
         0
          \end{array}
          \right)\cdotp
\end{equation}
  Next consider a continuous function $\eta:[0,\frac{\pi}{2}]\to \R,$ which is compactly supported in $(0,\frac{\pi}{2})$. Moreover, consider 
  \begin{align}\label{sinweakzero}
  \int\limits_{0}^{\frac{\pi}{2}} \sin (t\sin x_3 ) \eta(x_3)\,dx_3
   &=  \int\limits_{0}^{1} \sin (ty)\, \frac{\eta(\arcsin(y))}{\sqrt{1-y^2}}\,dy\to 0,~ {\rm as}~t\to \infty,
  \end{align}
where  we used the fact that   $\{\sin (t y)\}_{t>0}$ converges weakly to zero in the space $L^2([0, 1])$.  
Using the same arguments used to establish \eqref{sinweakzero}, one can conclude that $\{
\sin(\lambda x_1+x_2+t\cos x_3)\}_{t>0}$ weakly converges to zero in $L^2(\T^3)$ as $t\rightarrow\infty.$ Hence, $u^{E}(\cdot,t)$ converges weakly to $u^{K}$ in $L^2(\T^3)$ as $t\rightarrow\infty$. To see that $u^{E}(\cdot,t)$ does not converge strongly to $u^{K}$ in $L^2(\T^3)$ as $t\rightarrow\infty$, note that the $L^{2}$ norm of $u^{E}(\cdot,t)$ is conserved in time and is initially not equal to the $L^{2}$ norm of $u^{K}$.  

\subsection{2.75D shear flows and Onsager supercritical inviscid limits}\label{s-2.75limit}

The solution $u^\nu: \R_+\times \T^3\to \R^3$ defined by \eqref{3D-nu-solu} satisfies the following energy  equality
 \begin{equation}\label{EE}
 \|u^\nu(t, \cdot)\|_{L^2}^2+2\nu\int\limits_0^t \int\limits_{\T^3}|\nabla u^{\nu}(s, x)|^2\,dxds=\|u_{\rm in}\|_{L^2}^2, ~{\rm for ~all~}t\geq0.
 \end{equation}
Such solutions are unique in the class of 2.75D flows sharing the same symmetry, yet may not necessarily be unique in the general class of weak Leray-Hopf solutions with the same initial data in $L^2(\T^3)$.\footnote{By weak-strong uniqueness, 2.75D shear flows \eqref{3D-nu-solu} with $f\in L^p$, $p\geq 3$, and $g\in C^\infty$ are unique amongst the general class of weak Leray-Hopf solutions.} Recall from \eqref{3D-Euler-solu} that the corresponding Euler solution with initial data \eqref{3D-nu-data} is given by \begin{equation}\label{3D-eu}
u^{E}(t, x)=(f(\lambda x_1+x_2+tg(x_3), x_3), -\lambda f(\lambda x_1+x_2+tg(x_3), x_3)-g(x_3), 0).
\end{equation} In this section, we investigate properties of $u^{\nu}$, $u^{E}$ and the vanishing viscosity limit.   
 
\begin{proposition}\label{Th-4}
 Let $f\in L^2(\T^2; \R)$ and  $g\in{C}^\infty(\T).$  Consider initial data $u_{\rm in}$ in the form of  \eqref{3D-nu-data} with associated $f, g.$ 
Suppose that  $u^\nu$ is  the  global-in-time  solution  given by \eqref{3D-nu-solu} to the  problem \eqref{NS} ($\nu>0$) with initial data $u_{\rm in}$, and $u^E$ is  the  global-in-time  solution given by \eqref{3D-eu} for  the 3D Euler equations with initial data $u_{\rm in}$.

The above set up implies that the following holds true:
\begin{enumerate}[label=(\arabic*)]
\item $\forall ~T>0, u^\nu\to u^E$ in $L^2((0, T)\times\T^3)$ as $\nu\to0.$
\item (pointwise convergence) $\forall~t \geq 0, u^\nu(t, \cdot)\to u^E(t, \cdot)$ in $L^2(\T^3)$ as $\nu\to0.$  
\item There is an absence of anomalous dissipation in the vanishing viscosity limit. Namely, for all $t\geq0:$
\begin{equation}\label{Ab-diss}
\lim_{\nu\to 0}\,\nu\int\limits_0^t\int\limits_{\T^3}|\nabla u^\nu(s, x)|^2\,dxds=0.
\end{equation}
\item Let $f\in L^3(\T^2; \R)$ then $u^E$ satisfies the local energy balance. Namely for all positive $\varphi\in {C}^\infty(\R_+\times\T^3; \R)$ and $t\geq 0$
\begin{multline}\label{LEB}
\qquad\int\limits_{\T^3} |u^E(t, x)|^2\varphi(t, x)\,dx-\int\limits_{\T^3} |u_{\rm in}(x)|^2\varphi(0, x)\,dx\\
=\int\limits_0^t\int\limits_{\T^3}\partial_t\varphi|u^E|^2+|u^E|^2u^E\cdot\nabla\varphi\,dxds.
\end{multline}
\end{enumerate}
\end{proposition}

\begin{proof}[Proof of Proposition \ref{Th-4}]
We first prove item $(1).$ Following the same arguments as \cite[Theorem 5]{BTW}, we see that
\begin{equation}\label{2.75vv}
u^{\nu}\overset{\ast}{\rightharpoonup}  u^{E}\quad\textrm{in}\quad L^{\infty}(0,T; L^{2}(\mathbb{T}^3))\quad\mbox{as}\ \nu\rightarrow 0.
\end{equation}
Using arguments along the same lines as \cite{BT}, we have that for all $t\geq 0$, 
\begin{equation}\label{e.equalL2}
\|u^E(\cdot,t)\|_{L^2}^2=\|u_{\rm in}\|_{L^2}^2.
\end{equation}
Thanks to \eqref{2.75vv}, one has that
\begin{equation}\label{Th4-1}
u^\nu \rightharpoonup u^E~{\rm weakly}~{\rm in}~L^2((0, T)\times \T^3),~{\rm as}~\nu\to 0.
\end{equation}
Then it is enough to show 
\begin{align}\label{Th4-2}
\int\limits_0^T\int\limits_{\T^3} |u^\nu(s, x)|^2\,dxds\to \int\limits_0^T\int\limits_{\T^3}|u^E(s, x)|^2\,dxds,\quad {\rm as}~\nu\to0.
\end{align}
We have  the integrated energy balances:
 \begin{align*} 
   &\int\limits_0^T\int\limits_{\T^3}|u^\nu(t, x)|^2\,dxdt+2\nu\int\limits_0^T\int\limits_0^t \int\limits_{\T^3}|\nabla u^{\nu}(s, x)|^2\,dxdsdt=T\|u_{\rm in}\|_{L^2}^2,\\
   &\int\limits_0^T\int\limits_{\T^3}|u^E(t, x)|^2\,dxdt=T\|u_{\rm in}\|_{L^2}^2.
 \end{align*}
From the first energy balance above,
 \begin{multline*} 
 \limsup_{\nu\to0}\left( \int\limits_0^T\int\limits_{\T^3}|u^\nu(t, x)|^2\,dxdt \right)\\ \leq  \limsup_{\nu\to0}\left(\int\limits_0^T\int\limits_{\T^3}|u^\nu(t, x)|^2\,dxdt +2\nu\int\limits_0^T\int\limits_0^t \int\limits_{\T^3}|\nabla u^{\nu}(s, x)|^2\,dxdsdt   \right)=T\|u_{\rm in}\|_{L^2}^2.
 \end{multline*}
From \eqref{Th4-1} and \eqref{e.equalL2},
 \begin{align*} 
T\|u_{\rm in}\|_{L^2}^2= \int\limits_0^T\int\limits_{\T^3}|u^E(t, x)|^2\,dxdt \leq \liminf_{\nu\to0}\left( \int\limits_0^T\int\limits_{\T^3}|u^\nu(t, x)|^2\,dxdt \right).
 \end{align*}
Thus \eqref{Th4-2} is satisfied and we get $u^\nu\to u^E$ in $L^2((0, T)\times\T^3)$ as $\nu\to0.$

To prove item $(2),$ we proceed similarly as for item $(1).$ First, due to \eqref{3D-nu-solu}-\eqref{DDheat} and the energy equality \eqref{EE}, we have that
\begin{equation}\label{timenegder}
\sup_{\nu}\|\partial_{t} u^{\nu}\|_{L^{2}(0,T; W^{-1,2}(\mathbb{T}^3))}<\infty.
\end{equation}
This, together\footnote{We refer to \cite[page 104]{gregory2014lecture} where a similar argument is used.} with a classical diagonalisation argument argument and \eqref{2.75vv}, yields that 
\begin{equation}\label{weakconvpw}
\forall t\geq 0\quad u^{\nu}(\cdot,t)\rightharpoonup u^{E}(\cdot,t)\quad\textrm{in}\,\, L^{2}(\mathbb{T}^3).
\end{equation}
 Using \eqref{EE} again, we write
 \begin{multline*} 
 \limsup_{\nu\to0}  \int\limits_{\T^3}|u^\nu(t, x)|^2\,dx \\  \leq  \limsup_{\nu\to0} \bigg( \int\limits_{\T^3}|u^\nu(t, x)|^2\,dx+2\nu\int\limits_0^t \int\limits_{\T^3}|\nabla u^{\nu}(s, x)|^2\,dxds   \bigg)=\|u_{\rm in}\|_{L^2}^2.
 \end{multline*}
By virtue of \eqref{weakconvpw} and the energy balance \eqref{e.equalL2} for the associated Euler flow,
 we  write
 \begin{align*} 
 \|u_{\rm in}\|_{L^2}^2= \int\limits_{\T^3}|u^E(t, x)|^2\,dx\leq \liminf_{\nu\to0}   \int\limits_{\T^3}|u^\nu(t, x)|^2\,dx\leq \|u_{\rm in}\|_{L^2}^2.
 \end{align*}
Thus
\begin{align*} 
 \int\limits_{\T^3} |u^\nu(t, x)|^2\,dx \to  \int\limits_{\T^3} |u^E(t, x)|^2\,dx,\quad {\rm as}~\nu\to0,
\end{align*}
which together with \eqref{weakconvpw} gives item $(2).$

For the proof of item $(3),$ notice that for any $t\geq0,$
\begin{align*}
\int\limits_{\T^3}|u^\nu(t, x)|^2\,dx+2\nu\int\limits_0^t \int\limits_{\T^3}|\nabla u^{\nu}(s, x)|^2\,dxds =\|u_{\rm in}\|_{L^2}^2=\int\limits_{\T^3}|u^E(t, x)|^2\,dx.
\end{align*}
It is then easy to find that
\begin{align*}
\lim_{\nu\to0}\left(\int\limits_{\T^3}|u^\nu(t, x)|^2\,dx+2\nu\int\limits_0^t \int\limits_{\T^3}|\nabla u^{\nu}(s, x)|^2\,dxds\right) =\int\limits_{\T^3}|u^E(t, x)|^2\,dx.
\end{align*} 
Combining with item $(2)$ implies \eqref{Ab-diss}.

Let us now prove item $(4)$. The proof includes two steps: at first, we mollify $f$  to get a series of   smooth solutions in the form of \eqref{3D-Euler-solu} and write down the local energy equalities for these  smooth solutions, then we pass to the limit by using that $f\in L^3$  (which is sharp as an assumption in view of the nonlinear term).

Note that since $f\in L^3(\T^2; \R),$ there exists   sequence  $\{f_{k}\}_{k\in\mathbb{N}}\in{C}^\infty(\T^2; \R)$ such that $\|f_k-f\|_{L^3}\to0, ~k\to\infty.$  Define
\begin{equation*} 
u^{E, k}(t, x)=(f_k(\lambda x_1+x_2+tg(x_3), x_3), -\lambda f_k(\lambda x_1+x_2+tg(x_3), x_3)-g(x_3), 0).
\end{equation*}
By Fubini's theorem, one has
\begin{multline*} 
\|u^{E, k}(t, \cdot)-u^E(t, \cdot)\|_{L^3}^3\\\leq (|\lambda|^3+1)\int\limits_{\T^3}|f_k(\lambda x_1+x_2+tg(x_3), x_3)-f(\lambda x_1+x_2+tg(x_3), x_3)|^3\,dx\\
=2\pi(|\lambda|^3+1)\|f_k-f\|_{L^3}^3\to0,~{\rm as}~k\to\infty
\end{multline*}
and
\begin{align*}
\lim_{k\to \infty}\|u^{E, k}(0, \cdot)-u_{\rm in}( \cdot)\|_{L^2}=0.
\end{align*}
Now, since $u^{E, k}$ is smooth on $\R_+\times \T^3$ so for any positive $\varphi\in{C}^\infty(\R_+\times\T^3)$ and $t\geq0,$
 \begin{multline*} 
 \int\limits_{\T^3} |u^{E, k}(t, x)|^2\varphi(t, x)\,dx\\=\int\limits_{\T^3} |u^{E, k}(0, x)|^2\varphi(0, x)\,dx+\int\limits_0^t\int\limits_{\T^3}\partial_t\varphi|u^{E, k}|^2+|u^{E, k}|^2u^{E, k}\cdot\nabla\varphi\,dxds.
\end{multline*}
Using above convergence properties, we pass to the limit as $k\to\infty,$ and  obtain  the energy balance \eqref{LEB}.
\end{proof}
Note that the third item in Proposition \ref{Th-4} only gives convergence of energy pointwise. Below we show that if the Euler solution has some additional Onsager \emph{supercritical} Sobolev regularity, then one obtains a uniform convergence of the energy (i.e. $u^\nu\to u^E$ in $L^\infty(\R_+; L^2(\T^3))$) and a rate of vanishing of the anomalous dissipation. Proposition \ref{Th-4} also implies that the sufficient conditions in \cite{C86} and \cite{Masmoudi07} on the Euler flow for the inviscid limit to hold in $L^\infty(\R_+; L^2(\T^3))$ are not necessary  conditions. For 2D results on the vanishing viscosity and conservation of energy in Onsager supercritical regimes we refer to \cite{cheskidov2015onsager,lanthaler2021conservation}. 

In particular,  we have the following result. 
\begin{proposition}\label{Th-5}
 Let $f\in {H}^\ell(\T^2; \R)$ with $\ell\in(0, \frac{5}{6}).$    Consider initial data $u_{\rm in}$ in the form of  \eqref{3D-nu-data} with associated $f$ and $g\in C^{\infty}(\mathbb{T})$. 
Suppose that  $u^\nu$ is  the  global-in-time  solution  given by \eqref{3D-nu-solu} for the  problem \eqref{NS} ($\nu>0$) with initial data $u_{\rm in}$, and $u^E$ is  the  global-in-time  solution given by \eqref{3D-eu} for  the 3D Euler equations with initial data $u_{\rm in}$.

The above set up implies that the following holds true:
\begin{itemize}
\item[(1)] For $t\geq 0$ and $\nu\in (0,1)$,
$$\nu\int\limits_0^t\int\limits_{\T^3}|\nabla u^\nu|\,dxds\leq \nu^{\ell} C(\lambda,\ell,t,\|g'\|_{L^{\infty}},\|f\|_{H^{\ell}}).$$

\item[(2)] If $f\in {H}^\ell(\T^2; \R)$ with $\ell\in[\frac{1}{2}, \frac{5}{6})$, then 
\begin{align*}
u^\nu\to u^E\quad{\rm{in}}~~{C}([0, T]; L^2(\T^3)), ~\forall~T>0
\end{align*}
and for any $\ell'\in[0, \ell)$ and $\nu\in(0,1)$,
\begin{equation}\label{vvrate}
\sup_{t\in[0, T]} \|u^\nu(t,\cdot)-u^E(t,\cdot)\|_{\dot{H}^{\ell'}}\leq \nu^{\frac{1}{4}(1-\frac{\ell'}{\ell})} C(\lambda,\ell,\ell',T,\|g\|_{W^{3,\infty}},\|f\|_{H^{\ell}}).
\end{equation}
\end{itemize}
\end{proposition}

 \begin{proof}
 First let us establish item $(1)$. Fix $\nu>0$. For $f\in L^{2}(\mathbb{T}^2;\mathbb{R})$, let $T_{\nu}(f):= F_{\nu}$ be such that $F_{\nu}$ satisfies \eqref{DDheat}. Then 
 $$T_{\nu}:L^{2}(\mathbb{T}^2)\rightarrow L^{2}(0,\infty; \dot{H}^{1}(\mathbb{T}^2))\cap C([0,\infty; L^{2}(\mathbb{T}^2))$$ is a well-defined linear operator, due to the uniqueness of solutions to \eqref{DDheat} in the class $ L^{2}(0,\infty; \dot{H}^{1}(\mathbb{T}^2))\cap C([0,\infty; L^{2}(\mathbb{T}^2))$ with $L^{2}$ initial data. Furthermore, an $L^{2}$ energy estimate on \eqref{DDheat} yields
 \begin{equation}\label{TEEL2}
 \|T_{\nu}(f)\|_{L^{\infty}(0,\infty; L^{2})}, \nu^{\frac{1}{2}}\|T_{\nu}(f)\|_{L^{2}(0,\infty; \dot{H}^{1})}\leq \|f\|_{L^{2}}. 
 \end{equation}
 When $f\in H^{1}(\mathbb{T}^2)$, applying an $H^{1}$ energy estimate to \eqref{DDheat}  and then Gronwall's inequality yields 
 \begin{equation}\label{TEEH1}
  \|T_{\nu}(f)\|_{L^{\infty}(0,\infty; \dot{H}^{1})}, \nu^{\frac{1}{2}}\|T_{\nu}(f)\|_{L^{2}(0,\infty; \dot{H}^{2})}\lesssim e^{ct\|g'\|_{L^{\infty}}}\|f\|_{H^{1}}.
 \end{equation}
 Here, $c>0$ is a universal constant. Using \eqref{TEEL2}-\eqref{TEEH1}, we can apply \cite[Theorem 2.2.10]{HVVW16}, \cite[Theorem 6.45]{BL76} and the interpolation theory for linear operators \cite[Theorem 7.23]{AF03}. This yields that for every $f\in {H}^{\ell}(\mathbb{T}^2)$ ($\ell\in (0,\tfrac{5}{6})$) and $t\geq 0$
 \begin{equation}\label{HlL2est}
 \nu^{\frac{1}{2}}\|F_{\nu}\|_{L^{2}(0,t;\dot{H}^{1+\ell})}\lesssim e^{c\ell t\|g'\|_{L^{\infty}}}\|f\|_{H^{\ell}}.
 \end{equation}
 Applying similar arguments pointwise in time to $T_{\nu}(f)$ also yields that for every $f\in {H}^{\ell}(\mathbb{T}^2)$ ($\ell\in (0,\tfrac{5}{6})$) and $t\geq 0$
 \begin{equation}\label{HlLinfinftyest}
 \|F_{\nu}\|_{L^{\infty}(0,t;\dot{H}^{\ell})}\lesssim e^{c\ell t\|g'\|_{L^{\infty}}}\|f\|_{H^{\ell}}.
 \end{equation}
 Using the interpolation inequality for Sobolev spaces, H\"{o}lder's inequality and \eqref{HlL2est}-\eqref{HlLinfinftyest} gives that for any $f\in {H}^{\ell}(\mathbb{T}^2)$ ($\ell\in (0,\tfrac{5}{6})$) and $t\geq 0$:
 \begin{align}
 \nu \int\limits_{0}^{t}\int\limits_{\mathbb{T}^2} |\nabla F_{\nu}|^2 dxds&\leq \nu\int\limits_{0}^{t} \|F_{\nu}(s,\cdot)\|_{\dot{H}^{\ell}}^{2\ell}\|F_{\nu}(s,\cdot)\|_{\dot{H}^{1+\ell}}^{2(1-\ell)} ds\notag\\
 &\leq (\nu t)^{\ell} \|F_{\nu}\|_{L^{\infty}(0,t;\dot{H}^{\ell})}^{2\ell} \Big(\nu\int\limits_{0}^{t} \|F_{\nu}(s,\cdot)\|^2_{\dot{H}^{1+\ell}} ds\Big)^{1-\ell} \nonumber
 \\ 
 &\lesssim_{\ell} (\nu t)^{\ell}e^{2c\ell t\|g'\|_{L^{\infty}}}\|f\|_{H^{\ell}}^2. \label{H1int}
 \end{align} 
 Recall that $u^{\nu}$ is given by \eqref{3D-nu-solu}, where $\lambda\in\mathbb{Z}\setminus\{0\}.$  Together with \eqref{H1int}, this gives
 \begin{align*}
 \nu\int\limits_{0}^{t}\int\limits_{\T^3} |\nabla u^{\nu}|^2 dxds&\lesssim \lambda^4 \nu \int\limits_{0}^{t}\int\limits_{\T^3} |\nabla F_{\nu}|^2 dxds+\nu\int\limits_{0}^{t}\int\limits_{\T} |\nabla e^{\nu s\partial_{3}^2}g|^2 dxds \\
 &\lesssim_{\ell} (\nu t)^{\ell}e^{2c\ell t\|g'\|_{L^{\infty}}}\|f\|_{H^{\ell}}^2+t\nu \|g'\|_{L^{2}}^2.
 \end{align*}
 This establishes item $(1)$.
 
 Let us now prove item $(2)$. Define
 $$F(t,x_1,x_2):= f(x_1+tg(x_2), x_2),$$
 which is a distributional solution to
 \begin{equation}\label{Dtrans}
\left\{\begin{aligned}
  &\partial_t F  -g (y_2)\,\partial_1 F  = 0 \quad {\rm in} ~~\T^2\times\R_+ \\
 & F (0, y_1, y_2)=f(y_1, y_2).
\end{aligned}\right.
\end{equation}
Furthermore, by Fubini's theorem
\begin{equation}\label{L2Fest}
\|F(t,\cdot)\|_{L^{2}}=\|f\|_{L^2}\quad\forall t\geq 0.
\end{equation}
Similar arguments as those used to establish \eqref{HlLinfinftyest} give that for all $f\in H^{\ell}$ and $t\geq 0$
\begin{equation}\label{HlFest}
\|F\|_{L^{\infty}(0,t;\dot{H}^{\ell})}\lesssim \max(1, t\|g'\|_{L^{\infty}})^{\ell}\|f\|_{H^{\ell}}.
\end{equation}
Hence, using this and \eqref{HlLinfinftyest} gives
\begin{equation}\label{Hldiffest}
\|F_{\nu}-F\|_{L^{\infty}(0,t;\dot{H}^{\ell})}\lesssim (\max(1, t\|g'\|_{L^{\infty}})^{\ell}+e^{c\ell t\|g'\|_{L^{\infty}}})\|f\|_{H^{\ell}}
\end{equation}
Using \eqref{HlFest}, \eqref{L2Fest} and the interpolation inequality for Sobolev spaces, one deduces that
\begin{equation}\label{H1/2Fest}
\|F\|_{L^{\infty}(0,t;\dot{H}^{\frac{1}{2}})}\lesssim_{\ell}\max(1, t\|g'\|_{L^{\infty}})^{\frac{1}{2}}\|f\|_{H^{\ell}}.
\end{equation}
Similarly, \eqref{TEEL2} and \eqref{HlLinfinftyest} imply
\begin{equation}\label{H1/2Fnuest}
\|F_{\nu}\|_{L^{\infty}(0,t;\dot{H}^{\frac{1}{2}})}\lesssim_{\ell}e^{\frac{c t\|g'\|_{L^{\infty}}}{2}}\|f\|_{H^{\ell}}.
\end{equation}
Using the interpolation inequality for Sobolev spaces and H\"{o}lder's inequality, we have
\begin{align*}
\nu\int\limits_{0}^{t} \|F_{\nu}\|_{\dot{H}^{\frac{3}{2}}}^2 ds&\leq \nu\int\limits_{0}^{t} \|F_{\nu}\|_{\dot{H}^{\ell}}^{2(\ell-\frac{1}{2})}\|F_{\nu}\|_{\dot{H}^{1+\ell}}^{2(\frac{3}{2}-\ell)} ds\\
&\leq \|F_{\nu}\|_{L^{\infty}(0,t; \dot{H}^{\ell})}^{2(\ell-\frac{1}{2})}\Big(\nu\int\limits_{0}^{t} \|F_{\nu}\|_{\dot{H}^{1+\ell}}^2 ds\Big)^{\frac{3}{2}-\ell} (t\nu)^{\ell-\frac{1}{2}}.
\end{align*}
Using this and \eqref{HlL2est}-\eqref{HlLinfinftyest} gives
\begin{equation}\label{L2H3/2est}
\nu\int\limits_{0}^{t} \|F_{\nu}\|_{\dot{H}^{\frac{3}{2}}}^2 ds\leq e^{c\ell t\|g'\|_{L^{\infty}}}\|f\|_{H^{\ell}}^2(t\nu)^{\ell-\frac{1}{2}}.
\end{equation}
Next, we consider the equation satisfied by $F_{\nu}-F$:
 \begin{equation}\label{Diffeq}
\left\{\begin{aligned}
  \partial_t (F_{\nu}-F)  -g (y_2)\,\partial_1 (F_{\nu}-F)  =& (e^{t\nu\partial_{y_2}^2}g-g(y_2))\partial_{1} F_{\nu}\\
  &+{\nu}\bigl((\lambda^2+1)\partial_{1}^2+ \partial_{2}^2\bigr) F_{\nu} \quad {\rm in} ~~\T^2\times\R_+ \\
  (F_{\nu}-F) (0, y_1, y_2)=&0.
\end{aligned}\right.
\end{equation}
Performing an $L^{2}$ energy estimate\footnote{All subsequent estimates can be rigorously justified by approximating $f\in H^{\ell}$ by smooth $f_{k}\rightarrow f$ in $H^{\ell}$.} on \eqref{Diffeq} gives
\begin{align}
\|(F_{\nu}-F)(t,\cdot)\|_{L^{2}}^2=&2\int\limits_{0}^{t}\int\limits_{\mathbb{T}^2}(e^{t\nu\partial_{y_2}^2}g-g(y_2))\partial_{1} F_{\nu}(F_{\nu}-F) dy_1dy_2 ds \label{EEdiff}\\
&+2\int\limits_{0}^{t}\int\limits_{\mathbb{T}^2}{\nu}\bigl((\lambda^2+1)\partial_{1}^2+ \partial_{2}^2\bigr) F_{\nu}(F_{\nu}-F) dy_1dy_2 ds:=I+II.\nonumber
\end{align}
First, let us estimate $I$:
$$|I|\leq 2t^{\frac{1}{2}}\|e^{t\nu\partial_{y_2}^2}g-g\|_{L^{\infty}(\mathbb{T}\times (0,t))}\|F_{\nu}-F\|_{L^{\infty}(0,t; L^{2})}\Big(\int\limits_{0}^{t}\|\partial_{1} F_{\nu}\|_{L^{2}}^2 ds\Big)^{\frac{1}{2}}. $$
This, Lemma \ref{Le-heat}, \eqref{TEEL2} and \eqref{L2Fest} imply that
\begin{equation}\label{Iest}
I\lesssim t^{\frac{3}{2}}\nu^{\frac{1}{2}}\|f\|_{L^{2}}\|g''\|_{L^{\infty}}.
\end{equation}
Now we estimate II in \eqref{EEdiff}:
$$II\leq 2(\nu t)^{\frac{1}{2}}(\lambda^2+1)\|F_{\nu}-F\|_{L^{\infty}(0,t; \dot{H}^{\frac{1}{2}})}\Big(\nu\int\limits_{0}^{t} \|F_{\nu}\|_{\dot{H}^{\frac{3}{2}}}^2 ds\Big)^{\frac{1}{2}}. $$
Using \eqref{H1/2Fest}-\eqref{H1/2Fnuest} and \eqref{L2H3/2est} gives
\begin{equation}\label{IIest}
II\lesssim_{\ell}(\nu t)^{\frac{1}{4}+\frac{\ell}{2}}(\lambda^2+1)e^{\frac{c\ell t\|g'\|_{L^{\infty}}}{2}}\big(\max(1, t\|g'\|_{L^{\infty}})^{\frac{1}{2}}+e^{\frac{c t\|g'\|_{L^{\infty}}}{2}}\big)\|f\|_{H^{\ell}}.
\end{equation}
Combining \eqref{Iest}-\eqref{IIest} gives
\begin{align}
&\sup_{t\in [0,T]}\|F_{\nu}(t,\cdot)-F(t,\cdot)\|_{L^{2}}^2\nonumber\\
&\lesssim_{\ell}(\nu T)^{\frac{1}{4}+\frac{\ell}{2}}(\lambda^2+1)e^{\frac{c\ell T\|g'\|_{L^{\infty}}}{2}}\big(\max(1, T\|g'\|_{L^{\infty}})^{\frac{1}{2}}\nonumber\\
&\quad+e^{\frac{c T\|g'\|_{L^{\infty}}}{2}}\big)\|f\|_{H^{\ell}}+T^{\frac{3}{2}}\nu^{\frac{1}{2}}\|f\|_{L^{2}}\|g''\|_{L^{\infty}}\label{L2diff}.
\end{align}
Using this, the fact that $u^{\nu}$ and $u^{E}$ are given by \eqref{3D-nu-solu} and  \eqref{3D-eu} and Lemma \ref{Le-heat} , we get the following. Namely,
\begin{align}
&\sup_{t\in [0,T]}\|u^{\nu}(t,\cdot)-u^{E}(t,\cdot)\|_{L^{2}}^2\nonumber\\
&\lesssim_{\ell} \sup_{t\in [0,T]}\|e^{t\nu\partial_{y_2}^2}g-g\|_{L^{2}}^2+(\nu T)^{\frac{1}{4}+\frac{\ell}{2}}(\lambda^2+1)e^{\frac{c\ell T\|g'\|_{L^{\infty}}}{2}}\big(\max(1, T\|g'\|_{L^{\infty}})^{\frac{1}{2}}\nonumber\\
&\qquad+e^{\frac{c T\|g'\|_{L^{\infty}}}{2}}\big)\|f\|_{H^{\ell}}+T^{\frac{3}{2}}\nu^{\frac{1}{2}}\|f\|_{L^{2}}\|g''\|_{L^{\infty}}\nonumber\\
&\lesssim (\nu T)^2\|g''\|_{L^{\infty}}+(\nu T)^{\frac{1}{4}+\frac{\ell}{2}}(\lambda^2+1)e^{\frac{c\ell T\|g'\|_{L^{\infty}}}{2}}\big(\max(1, T\|g'\|_{L^{\infty}})^{\frac{1}{2}}\nonumber\\
&\qquad+e^{\frac{c T\|g'\|_{L^{\infty}}}{2}}\big)\|f\|_{H^{\ell}}+T^{\frac{3}{2}}\nu^{\frac{1}{2}}\|f\|_{L^{2}}\|g''\|_{L^{\infty}}.\label{vvL2rate}
\end{align}
This gives \eqref{vvrate} for $\ell'=0$ as required. To get \eqref{vvrate} for $\ell'\in (0,\ell)$ we interpolate \eqref{L2diff} with \eqref{Hldiffest} to get
\begin{equation}\label{Hl'diff}
\sup_{t\in[0, T]} \|F_{\nu}(t,\cdot)-F(t,\cdot)\|_{\dot{H}^{\ell'}}{\leq \nu^{\frac{1}{4}(1-\frac{\ell'}{\ell})} C(\ell,\ell',T,\|g\|_{W^{2,\infty}},\|f\|_{H^{\ell}}) }.
\end{equation}
Furthermore, using Lemma \ref{Le-heat} we see that
\begin{equation}\label{Hl'heat}
\sup_{t\in [0,T]}\|e^{t\nu\partial_{y_2}^2}g-g\|_{\dot{H}^{\ell'}}\lesssim (\nu T)\|g\|_{W^{3,\infty}}.
\end{equation}
By similar reasoning as the $\ell'=0$ case we then get
\begin{multline}
\sup_{t\in [0,T]}\|u^{\nu}(t,\cdot)-u^{E}(t,\cdot)\|_{\dot{H}^{\ell'}}\\
\leq |\lambda|^{1+\ell'}\sup_{t\in [0,T]} \|F_{\nu}(t,\cdot)-F(t,\cdot)\|_{\dot{H}^{\ell'}}+\sup_{t\in [0,T]}\|e^{t\nu\partial_{y_2}^2}g-g\|_{\dot{H}^{\ell'}}.
\label{vvsobrate1}
\end{multline}
Combing this with \eqref{Hl'diff}-\eqref{Hl'heat} gives \eqref{vvrate} for all $\ell'\in [0,\ell)$ as required. 
 \end{proof}
\subsection{Strong ill-posedness for  3D Euler equations in anisotropic spaces}

Let   $f\in W^{1, p}(\T; W^{1,q}(\T))$, $g\in W^{1,q}(\T)$ and $p,q\in [2,\infty)$. For $\lambda\in\Z\setminus\{0\}$, consider the following initial data 
 $$ ( f(\lambda x_1+x_2, x_3),  -\lambda f(\lambda x_1+x_2, x_3)-g(x_3), 0 )\in W^{1,p}(\T^2; W^{1,q}(\T)), $$
 which generates an explicit  solution\footnote{The fact that this is a weak solution to the 3D Euler equations uses the same arguments as in \cite{BT}[Theorem 2]} $U^{E}$ in the form of \eqref{3D-Euler-solu} to the 3D Euler equations on the torus. 
Using identical reasoning as in \cite{BT07}, it is clear that the roughness of $g$ means that $u_E(t, x)$ will not lie in the expected solution space $W^{1,p}(\T^2; W^{1,q}(\T))$ for any positive time $t$, which shows strong ill-posedness in the sense of Hadamard (non-existence)  in the anisotropic Sobolev space. We also anticipate that it is possible to show illposedness of the 3D Euler equations for initial data that has dependence on three spatial dimensions and belongs to other anisotropic spaces (analogous to the isotropic spaces considered in \cite{BT}).

\appendix

\section{Heuristics for the structure of 2.75D shear flows}
\label{S3a}

In this appendix we give some heuristics about the derivation of 2.75D shear flows. As mentioned earlier, they are rotated versions of the parallel flows introduced by Wang \cite{wang2001kato}. Here we outline another derivation based on the analysis of the following reduced Navier-Stokes system
\begin{equation}\label{2.75ns}
\left\{\begin{aligned}
 \partial_t{ u^{\rm h}}+u^{\rm h}\cdot\nabla_{\rm h} u^{\rm h}+\nabla_{\rm h} \mathnormal P&=\Delta \mathnormal u^{\rm h},\\
 \partial_3  P&=0,\\
 \div_{\rm h}\mathnormal u^{\rm h}&=0,\\
 u^{\rm h}|_{t=0}&=u^{\rm h}_{\rm in}.
\end{aligned}\right..
\end{equation}
We dub that system the `2.75D Navier-Stokes equations'.\footnote{Once we get a solution for system \eqref{2.75ns}, then it also satisfies the so-called primitive equations, see for example the works of Cao and Titi \cite{CT}  and  Hieber and  Kashiwabara \cite{HK16}  on primitive equations.} It is well-known that solutions to this system with $H^1$ data are smooth.\footnote{A byproduct of Theorem \ref{Th-2} is that system \eqref{2.75ns} is ill-posed for generic data. The proof is by contradiction. In fact, if for any initial data $u_{\rm in}^{\rm h}$ satisfying condition \eqref{Structure-initialdata}, there always exists a  solution $ u^{\rm h}$ to the problem \eqref{2.75ns} on some time interval $[0, T]$,  then one  can extend  $u^{\rm h}$ to  a solution $( u^{\rm h}, 0)$ for  the 3D Navier-Stokes  problem \eqref{NS}.  By local well-posedness  theory for \eqref{NS}, regularity results in \cite{NP99} and weak-strong uniqueness, one confirms that  $u=(u^{\rm h}, 0)$ is the unique solution on  $[0, T]$  supplemented with initial data $(u^{\rm h}_{\rm in}, 0)$.  In particular, it   implies  that $u^3 \equiv 0$ will be preserved.}

Consider initial data $u^{\rm h}_{\rm in}$ of the form  
\begin{equation}\label{in-2.75-1}
u^{\rm h}_{\rm in}=\Bigl(\partial_2\phi(x), \,-\partial_1\phi(x)\Bigr).
\end{equation}
Let us look for solutions of system \eqref{2.75ns} under the form
 \begin{equation}\label{solution-2.75-1}
 u^{\rm h}=\Bigl(\partial_2\Phi(t, x), \,-\partial_1\Phi(t, x)\Bigr)=\nabla_{\rm h}^{\perp} \Phi,
 \end{equation}
where $\nabla_{\rm h}^{\perp}:=(\partial_2, -\partial_1)$. Notice that  the pressure is now given by  
\begin{equation*} 
\left\{\begin{aligned}
&  \Delta_{\rm h} P =-\div_{\rm h}(u^{\rm h}\cdot\nabla_h u^{\rm h})=2   \det\bigl({\rm Hessian}_{\rm h}~~\Phi \bigr)\\
&  \partial_3 P=0,
\end{aligned}\right.
\end{equation*}
 where  we   used the vector identities
\begin{align*}
\div_{\rm h} \Bigl((\nabla_{\rm h}^{\perp} \Phi)\cdot\nabla_{\rm h} (\nabla_{\rm h}^\perp \Phi)\Bigr)&= \div_{\rm h}\div_{\rm h}\Bigl((\nabla_{\rm h}^{\perp} \Phi)  \otimes (\nabla_{\rm h}^{\perp} \Phi) \Bigr) \\
&= (\nabla_{\rm h}\nabla_{\rm h}^\perp \Phi): (\nabla_{\rm h}\nabla_{\rm h}^\perp \Phi)^{\rm T}=-2 \det\bigl({\rm Hessian}_{\rm h}~~\Phi\bigr),
\end{align*}
where 
\begin{equation*} 
 {\rm Hessian}_{\rm h}{:=}\left(  \begin{array}{cc}
         \partial_{1}^2 & \partial_{1}\partial_{2}\\
         \partial_{1}\partial_{2}  & \partial_{2}^2
          \end{array}
          \right)\cdotp
\end{equation*}
In order to have $\partial_3 P=0$, one has to satisfy
\begin{align*}\label{CONDITION}
\partial_3\det\bigl({\rm Hessian}_{\rm h}~~\Phi \bigr)=0.
\end{align*}
If there exist a function $\Psi: \T\times\R_+\to \R$  and a constant $\lambda\in \mathbb{Z}$ such that 
\begin{equation}\label{structure-2.75-solu}
\mathcal{L}_\lambda \Phi(t, x)=\Psi(t, x_3) \quad{\rm with}\quad  \mathcal{L}_\lambda := \partial_1-\lambda \partial_2,
\end{equation}
then we have
\begin{equation*}
\partial_1 \mathcal{L}_\lambda \Phi= \partial_2 \mathcal{L}_\lambda \Phi=0  \quad{\rm i.e.}\quad
\left(  \begin{array}{c}
         \partial_{1}^2 \Phi  \\
         \partial_{1}\partial_{2}  \Phi
          \end{array}
          \right)=\lambda \left( \begin{array}{c}
         \partial_{1}\partial_{2} \Phi  \\
         \partial_{2}^2  \Phi
          \end{array}\right),
\end{equation*}
and thus
\begin{equation} \label{CONDITION}
  \det\bigl({\rm Hessian}_{\rm h}~~\Phi \bigr)=0.
\end{equation}
In the following, we will focus on the case \eqref{structure-2.75-solu} for the  Cauchy problem \eqref{2.75ns}. Concerning the initial data, we also need to look for  a function $\psi:  \T\to \R$  such that
 \begin{equation}\label{structure-2.75-initial}
 \mathcal{L}_{\lambda}\phi(x) -\psi(x_3)=0.
 \end{equation}
Recalling that the   velocity $u^{\rm h}=\nabla^\perp_{\rm h} \Phi$ and taking into consideration \eqref{structure-2.75-solu},  one has
$$u^{\rm h}\cdot\nabla_{\rm h} u^{\rm h}=(-1, \lambda)\,\Psi \partial_{2}^2\Phi$$ 
and $P$ is a constant. 
Finally,  we are lead to considering the following system 
  \begin{equation*} 
\left\{\begin{aligned}
& \partial_t{\partial_2 \Phi}-\Psi(t, x_3)\, \partial_{2}^2\Phi   = \Delta \partial_2 \Phi\quad&\hbox{in}\ ~~\R_+\times\T^3,\\
& \partial_t{(\lambda \partial_2 \Phi+\Psi(t, x_3))}-\lambda \Psi(t, x_3)\, \partial_{2}^2\Phi  = \Delta (\lambda\partial_2 \Phi+\Psi(t, x_3)) \quad&\hbox{in}\ ~~\R_+\times\T^3,\\
&(\Phi, \Psi)|_{t=0}= (\phi, \psi)  \quad{\rm with}\quad  \mathcal{L}_{\lambda}\phi-\psi(x_3)=0,
\end{aligned}\right.
\end{equation*}
which can be simplified  as
\begin{equation}\label{2.75ns-S}
\left\{\begin{aligned}
& \partial_t{\partial_2 \Phi}-\Psi(t, x_3)\, \partial_{2}^2\Phi   = \Delta \partial_2 \Phi  \quad&\hbox{in}\ ~~\R_+\times\T^3,\\
& \partial_t\Psi(t, x_3)  = \partial_{3}^2\Psi(t, x_3)\quad&\hbox{in}\ ~~\R_+\times\T,\\
&(\Phi, \Psi)|_{t=0}= (\phi, \psi)  \quad{\rm with}\quad  \mathcal{L}_{\lambda}\phi-\psi(x_3)=0.
\end{aligned}\right.
\end{equation}
In conclusion, $\Psi(t, x_3)= (\mathcal{K}\star\psi) (t, x_3)$, where $\mathcal K$ is the one-dimensional heat kernel see \eqref{e.heat1d}, and $\partial_2 \Phi$ satisfies the  \emph{linear} transport-heat equation 
 \begin{equation} \label{eq-partial_2 Phi}
\left\{\begin{aligned}
& \partial_t v + V\cdot\nabla v = \Delta v\quad {\rm in} ~~\R_+\times\T^3 \quad  {\rm with}\quad  V=(0, -\Psi(t, x_3), 0),\\
& v_{\rm in}=\partial_2 \phi.
\end{aligned}\right.
\end{equation}
Taking $\psi(x_3)=g(x_3)$ and $$\phi(x)=\phi(\lambda x_1+ x_2, x_3)=\int\limits_0^{\lambda x_1+x_2} f(y_1, x_3)\,d y_1+ x_1 \psi(x_3).$$
Obviously, $\phi(x)$ and $\psi(x_3)$ satisfy \eqref{structure-2.75-initial}, so the associated 
solution  $\partial_2 \Phi(t, x)$ of \eqref{eq-partial_2 Phi} for which $u^{\rm h}$ given by \eqref{solution-2.75-1} solves problem \eqref{2.75ns}.

\section*{Acknowledgment}
The authors thank Hao Jia for stimulating discussions concerning the inviscid damping for the 2.75D shear flows, {as well as Jean-Yves Chemin for a discussion relating to Theorem \ref{Th-3}}. We also thank Pierre Gilles Lemari\'e-Rieusset and Helena J. Nussenzveig Lopes for their comments on an earlier version of this paper.

 CP and JT are partially supported by the Agence Nationale de la Recherche,
project BORDS, grant ANR-16-CE40-0027-01. CP is also partially supported by the Agence Nationale de la Recherche,
 project SINGFLOWS, grant ANR-
18-CE40-0027-01, project CRISIS, grant ANR-20-CE40-0020-01, by the CY
Initiative of Excellence, project CYNA (CY Nonlinear Analysis) and project CYFI (CYngular Fluids and Interfaces). JT is also supported by the Labex MME-DII. TB and CP thank the Institute of Advanced Studies of Cergy Paris University for their hospitality. 

\section*{Data availability statement}

Data sharing is not applicable to this article as no datasets were generated or analyzed during the current study.

\section*{Conflict of interest}

The authors declare that they have no conflict of interest.

\small 
\bibliographystyle{abbrv}
\bibliography{breaking.bib}

\end{document}